\documentclass[12pt,a4paper]{article}%
\usepackage{amsmath}
\usepackage{amsfonts}
\usepackage{mitpress}
\usepackage{amssymb}
\usepackage{graphicx}
\usepackage{hyperref}
\usepackage{color}
\usepackage{float,color,ulem}
\usepackage{epsf,epsfig,subfigure}
\usepackage{bm,soul}%
\providecommand{\U}[1]{\protect\rule{.1in}{.1in}}

\newfloat{figure}{H}{lof}
\newfloat{table}{H}{lot}
\floatname{figure}{\figurename}
\newtheorem{theorem}{Theorem}[section]

\newtheorem{corollary}[theorem]{Corollary}

\newtheorem{definition}{Definition}[section]

\newtheorem{lemma}{Lemma}[section]

\newtheorem{proposition}{Proposition}[section]
\newtheorem{remark}{Remark}[section]

\newenvironment{proof}[1][Proof]{\noindent\textbf{#1.} }{\ \rule{0.5em}{0.5em}}
\numberwithin{equation}{section}
\newdimen\dummy
\dummy=\oddsidemargin
\begin{document}
\title{Asymptotic stability of a nonlinear Korteweg-de Vries equation with a critical length}
\author{Jixun Chu\thanks{Department of Applied Mathematics, School of Mathematics and Physics, University of Science and Technology Beijing,
Beijing 100083, China, Universit\'{e} Pierre et Marie Curie-Paris 6,
UMR 7598 Laboratoire Jacques-Louis Lions, 75005 Paris, France.
E-mail: \texttt{chujixun@mail.bnu.edu.cn}. JXC was supported by the
ERC advanced grant 266907 (CPDENL) of the 7th Research Framework
Programme (FP7).}, Jean-Michel Coron\thanks{Institut universitaire
de France and Universit\'{e} Pierre et Marie Curie-Paris 6, UMR 7598
Laboratoire Jacques-Louis Lions, 75005 Paris, France. E-mail:
\texttt{coron@ann.jussieu.fr}. JMC was supported by the ERC advanced
grant 266907 (CPDENL) of the 7th Research Framework Programme
(FP7).}, Peipei Shang\thanks{Department of Mathematics, Tongji
University, Shanghai 200092, China. E-mail:
\texttt{peipeishang@hotmail.com}. PS was partially supported by the
ERC advanced grant 266907 (CPDENL) of the 7th Research Framework
Programme (FP7).}} \maketitle

\begin{abstract}
We study an initial-boundary-value problem of a nonlinear Korteweg-de Vries
equation posed on a finite interval $(0,2\pi).$ The whole system has Dirichlet
boundary condition at the left end-point, and both of Dirichlet and Neumann
homogeneous boundary conditions at the right end-point. It is known that the
origin is not asymptotically stable for the linearized system around the
origin. We prove that the origin is (locally) asymptotically stable for the
nonlinear system.

\vspace{0.2in}\noindent\textbf{Key words:} nonlinearity, Korteweg-de Vries
equation, stability, center manifold

\vspace{0.1in}\noindent\textbf{2000 MR Subject Classification:}\quad35Q53, 35B35

\end{abstract}

\section{Introduction}

\label{sec1}

\bigskip This article is concerned with the following initial-boundary-value
problem of the Korteweg-de Vries (KdV) equation posed on a finite interval
\begin{equation}
\left\{
\begin{array}
[c]{l}%
y_{t}+y_{x}+yy_{x}+y_{xxx}=0,\\
y(t,0)=y(t,L)=0,\\
y_{x}(t,L)=0,\\
y(0,x)=y_{0}\in L^{2}\left(  0,L\right)  ,
\end{array}
\right.  \label{0}%
\end{equation}
with $L=2\pi.$

The KdV equation was first derived by Boussinesq in
\cite{1877-Boussinesq} (see, in particular, equation (283 bis), p.
360) and Korteweg and de Vries in \cite{deJager} in order to
describe the propagation of small amplitude long water waves in a
uniform channel. This equation is now commonly used to model
unidirectional propagation of small amplitude long waves in
nonlinear dispersive systems.

Since in many physical applications the region is finite, people are
also interested in properties of the KdV equations on a finite
spacial domain. Moreover, Bona and Winther pointed out in
\cite{Bona-Winther} that the term $y_{x}$ should be incorporated in
the KdV equations to model the water waves when $x$ denotes the
spatial coordinate in a fixed frame. We refer to
\cite{Bona-Sun-Zhang, Bona-Sun-Zhang2009, Colin-Ghidaglia,
 Faminskii, Goubet-Shen, Homler, Kramer-Zhang, Rivas-Usman-Zhang} for
the well-posedness results of initial-boundary-value problems of the
KdV equations posed on a finite interval. From control theory point
of view, we refer to \cite{Cerpa-2012, Rosier-Zhang2009} for an
overall review and recent progress on different kinds of KdV
equations. In particular, when the spacial domain is of finite
interval, we refer to \cite{Cerpa2007, Coron-Crepeau, Crepeau,
Glass-Guerrero, Rosier1997, Rosier, Zhang} for the controllability
and \cite{Cerpa-Coron-preprint, Jia-Zhang2012,
Massarol-Menzala-Pazoto, Pazoto, Menzala-Vasconcellos-Zuazua} for
some stabilization results. We refer to
\cite{Cerpa-Crepeau-Stabilization, Komornik2, Komornik3,
Laurent-Rosier-Zhang, Russell-Zhang3, Russell-Zhang, Sun} for
studies on the KdV equations with periodic boundary conditions.

Rosier introduced in \cite{Rosier1997} the following set of critical lengths
\[
\mathcal{N}:=\left\{  2\pi\sqrt{\frac{j^{2}+l^{2}+jl}{3}};j,l\in
\mathbb{N}^{\ast}\right\}
\]
for the following KdV control system%
\begin{equation}
\left\{
\begin{array}
[c]{l}%
y_{t}+y_{x}+yy_{x}+y_{xxx}=0,\\
y(t,0)=y(t,L)=0,\\
y_{x}(t,L)=u(t),\\
y(0,x)=y_{0},
\end{array}
\right.  \label{KDV control}%
\end{equation}
where $u(t)\in\mathbb{R}$ is the control. We refer to
\cite{Cerpa-Crepeau, Coron-Crepeau, Rosier1997} for the
well-posedness and controllability of system \eqref{KDV control}.
Especially, Rosier proved in \cite{Rosier1997} that (\ref{KDV
control}) is locally controllable around the origin by analyzing the
corresponding linearized system and by means of Banach fixed point
theorem, provided that the spacial domain is not critical, i.e.
$L\notin\mathcal{N}$. However, this method does not work when
$L\in\mathcal{N}$, since the corresponding linearized system of
(\ref{KDV control}) around the origin is not any more controllable
in this case. By using the ``power series expansion'' method, Coron
and Cr\'{e}peau in \cite{Coron-Crepeau} obtained the local exact
controllability around the origin of the nonlinear KdV equation
\eqref{KDV control} with the critical length $L=2k\pi$ (i.e. taking
$j=l=k$ in $\mathcal{N}$), provided that (see \cite[Theorem 8.1 and
Remark 8.2]{CoronBook})
\begin{equation}
\left(  j^{2}+l^{2}+ jl=3k^{2}\text{ and }(j,l)\in\mathbb{N}\setminus
\{0\}^{2}\right)  \Rightarrow\left(  j=l=k\right)  .
\end{equation}
The cases with the other critical lengths have been studied by Cerpa in
\cite{Cerpa2007} and by Cerpa and Cr\'{e}peau in \cite{Cerpa-Crepeau} with the
same method, where the authors have proved that the nonlinear term $yy_{x}$
gives the local exact controllability around the origin.

If $L\notin\mathcal{N}$, it is proved by Perla Menzala, Vasconcellos
and Zuazua in \cite{Menzala-Vasconcellos-Zuazua} that $0$ is
exponentially stable for the linearized equation (\ref{linearized})
\begin{equation}
\left\{
\begin{array}
[c]{l}%
y_{t}+y_{x}+y_{xxx}=0,\\
y(t,0)=y(t,L)=0,\\
y_{x}(t,L)=0,\\
y(0,x)=y_{0}\in L^{2}\left(  0,L\right) ,
\end{array}
\right.  \label{linearized}%
\end{equation}
of (\ref{0}) around 0. Furthermore, it is also proved in
\cite{Menzala-Vasconcellos-Zuazua} that 0 is locally asymptotically
stable for system (\ref{0}). However, when $L\in\mathcal{N}$, it has
been proved by Rosier in \cite{Rosier1997} that (\ref{linearized})
admits a family of non-trivial solutions of the form $e^{\lambda
t}v_{\lambda}(x)$ for some $\lambda\in i\mathbb{R},$ where
$v_{\lambda}\in C^{\infty}([0,L])\setminus\{0\}$ satisfies
\[
\left\{
\begin{array}
[c]{c}%
\lambda v_{\lambda}(x)+v_{\lambda}^{\prime}(x)+v_{\lambda}^{\prime\prime
\prime}(x)=0,\\
v_{\lambda}(0)=v_{\lambda}(L)=v_{\lambda}^{\prime}(0)=v_{\lambda}^{\prime
}(L)=0.
\end{array}
\right.
\]
 For these critical lengths, it is therefore interesting to study the
influence of the nonlinear term $yy_{x}$ on the local asymptotic
stability of $0$ for the nonlinear KdV equation (\ref{0}). This
article is concerned with the stability property for system
(\ref{0}) with special critical length $L=2\pi$. In this particular
case, by Remark 3.6 of \cite{Rosier1997}, $a(1-\cos x)$,
$a\in\mathbb{R}$ are steady solutions of (\ref{linearized}).

Center manifolds play an important role in studying nonlinear
systems. We refer to
\cite{Carr1981,Chen-Hale-Tan,Haragus-Iooss2011,Magal-Ruan,Minh-Wu2004}
and the references therein for center manifold theories on abstract
Cauchy problems in Banach spaces. The authors in
\cite{Carr1981,Haragus-Iooss2011,Magal-Ruan} investigated directly
the evolution equations and gave some sufficient conditions for the
existence and smoothness of center manifolds. While, the authors in
\cite{Chen-Hale-Tan} presented a general result on the invariant
manifolds together with associated invariant foliations of the state
space, which can be applied directly to $C^{1}$ semigroups in Banach
space. But the method presented in \cite{Chen-Hale-Tan} has no
extension to the case of $C^{k}$-smoothness with $k>1.$ In
\cite{Minh-Wu2004}, by using the method of graph transforms, some
classical results about smoothness of invariant manifolds for maps
and the technique of ``lifting'', the existence, smoothness and
attractivity of invariant manifolds for evolutionary process on
general Banach spaces are proved when the nonlinear perturbation has
a small global Lipschitz constant and is locally $C^{k}$-smooth near
the trivial solution. Because of the existence of the nonlinear term
in \eqref{0}, the results presented in \cite{Carr1981,Magal-Ruan} do
not work for our system. Moreover, due to the fact that the linear
operator in our system \eqref{0} with $L=2\pi$ does not satisfy the
resolvent estimates provided by \cite{Haragus-Iooss2011}, we cannot
apply directly the results given in \cite{Haragus-Iooss2011}. Thanks
to the center manifold results given in \cite{Minh-Wu2004}, in this
article, we show the existence and smoothness of a center manifold
of (\ref{0}) with $L=2\pi$, and obtain that the stability property
can be determined by a reduced system of dimension one. Furthermore,
by studying the stability on this reduced one dimensional system, we
obtain the local asymptotic stability of $0$ for the original system
(\ref{0}) when $L=2\pi$. The main result of this article is the
following theorem.

\begin{theorem}
\label{main result}Let us assume that $L=2\pi$. Then  $0\in
L^{2}(0,L)$ is (locally) asymptotically stable for the nonlinear KdV equation (\ref{0}). More precisely:

\begin{itemize}
\item[(i)] For every $\varepsilon>0$, there exists $\delta=\delta
(\varepsilon)>0$ such that, if $\Vert
y_{0}\Vert_{L^{2}(0,L)}<\delta$,  then
\[
\Vert y(t,\cdot)\Vert_{L^{2}(0,L)}<\varepsilon,\quad\forall t\geq0.
\]

\item[(ii)] There exists $\delta_{1}>0$ such that, if $\Vert y_{0}\Vert
_{L^{2}(0,L)}<\delta_{1}$,  then
\[
\lim_{t\rightarrow+\infty}\Vert y(t,\cdot)\Vert_{L^{2}(0,L)}=0.
\]

\end{itemize}
\end{theorem}

\begin{remark}
The existence of $\delta(\varepsilon)$ is trivial and well known. In fact, one
can take $\delta(\varepsilon)=\varepsilon$ since $t\in[0,+\infty)\mapsto\Vert
y(t,\cdot)\Vert_{L^{2}(0,2\pi)}$ is nonincreasing (see also Lemma
\ref{decreasingL2norm} below). The nontrivial part of Theorem
\ref{main result} is property (ii).
\end{remark}

The organization of this paper is as follows: First, in Section~\ref{sec2},
some basic properties of the linearized system \eqref{linearized} are given.
Then, in Section~\ref{sec3}, we prove some properties of a non local
modification of the KdV equation \eqref{0} and then deduce the existence and smoothness
of the center manifold. Finally, in Section~\ref{sec4}, we analyze the dynamic
on the center manifold, which concludes the proof of the main result, i.e.
Theorem \ref{main result}.

\section{Preliminary}

\label{sec2}

In this section, we give some properties for the linearized system
(\ref{linearized}) with $L=2\pi$.

Set $X:=L^{2}\left(  0,L\right)  $. Let $A:D\left( A\right)
\rightarrow X$ be the linear operator defined by
\[
A\varphi=-\varphi_{x}-\varphi_{xxx}%
\]
with
\[
D(A)=\left\{  \varphi\in H^{3}\left(  0,L\right)  :\varphi\left(  0\right)
=\varphi\left(  L\right)  =\varphi_{x}\left(  L\right)  =0\right\}  .
\]
It is easily verified that both $A$ and its adjoint $A^{\ast}$ are
dissipative. The following proposition follows from \cite[Corollary. 4.4,
Chapter 1 ]{Pazy}. See also \cite{Rosier1997}.

\begin{proposition}
\label{PROP1}$A$ generates a $C_{0}$-semigroup of contractions on
$L^{2}(0,L).$
\end{proposition}

 From now on, we denote by $\left\{  S\left(  t\right)  \right\}  _{t\geq0}$
the $C_{0}$-semigroup associated with $A$. Then $S(t)y_{0}$ is the mild
solution of the linearized system \eqref{linearized} for any given initial
data $y_{0}\in L^{2}\left(  0,L\right)  $. By Proposition \ref{PROP1}, we
obtain the following lemma directly.

\begin{lemma}
\label{LE3}For every $y_{0}\in L^{2}\left(  0,L\right)  ,$ we have
\[
\left\Vert S(t)y_{0}\right\Vert _{L^{2}(0,L)}\leq\left\Vert y_{0}\right\Vert
_{L^{2}\left(  0,L\right)  },\quad \forall t\geq0.
\]

\end{lemma}

 Furthermore, the following Kato smoothing effect is given by Rosier
\cite[Proposition 3.2]{Rosier1997}.

\begin{lemma}
\label{LE1}For every $y_{0}\in L^{2}\left(  0,L\right)  $ and for every $T>0,$
we have $S(t)y_{0}\in L^{2}\left(  0,T;H^{1}(0,L)\right)  $ and
\[
\left\Vert S(t)y_{0}\right\Vert _{L^{2}\left(  0,T;H^{1}(0,L)\right)  }%
\leq\left(  \frac{4T+L}{3}\right)  ^{\frac{1}{2}}\left\Vert y_{0}\right\Vert
_{L^{2}\left(  0,L\right)  }.
\]

\end{lemma}

Proceeding as in \cite{Pazoto-Rosier}, we can prove the following two results.

\begin{lemma}
\label{LE2}There exists a constant $C>0$ such that for any $y_{0}\in H_{0}%
^{1}\left(  0,L\right)  ,$ the solution $S(t)y_{0}$ of (\ref{linearized})
fulfills
\[
\left\Vert S(t)y_{0}\right\Vert _{H_{0}^{1}(0,L)}\leq C\left\Vert
y_{0}\right\Vert _{H_{0}^{1}\left(  0,L\right)  },\quad \forall t\geq0.
\]

\end{lemma}

\begin{proof}
 For any $U_{0}\in D\left(  A\right)  $, let us define $U(t):=S(t)U_{0}.$ Let
$V\left(  t\right)  =U_{t}\left(  t\right)  =AU\left(  t\right)  .$ Then $V$
is the mild solution of the system%
\[
\left\{
\begin{array}
[c]{l}%
V_{t}=AV,\\
V(0)=AU_{0}\in L^{2}\left(  0,L\right)  .
\end{array}
\right.
\]
Hence, it follows from Lemma~\ref{LE3} that
\[
\left\Vert V\left(  t\right)  \right\Vert _{L^{2}(0,L)}\leq\left\Vert
V(0)\right\Vert _{L^{2}\left(  0,L\right)  },\quad \forall t\geq0.
\]
Since $V(t)=AU(t),V\left(  0\right)  =AU_{0},$ and the norms $\left\Vert
U\right\Vert _{L^{2}\left(  0,L\right)  }+\left\Vert AU\right\Vert
_{L^{2}\left(  0,L\right)  }$ and $\left\Vert U\right\Vert _{D\left(
A\right)  }$ are equivalent on $D(A),$ we conclude that, for some constant
$C_{1}>0$ independent of $U_{0}$ and $t\geq0$, we have
\[
\left\Vert U\left(  t\right)  \right\Vert _{D\left(  A\right)  }\leq
C_{1}\left\Vert U_{0}\right\Vert _{D\left(  A\right)  }.
\]
Then the result of Lemma~\ref{LE2} follows by a standard interpolation argument.
\end{proof}

Our next proposition shows that $\left\{  S\left(  t\right)  \right\}
_{t\geq0}$ is a compact semigroup.

\begin{proposition}
\label{PROP2}Let $T>0$. There exists a constant $C>0$ such that, for every
$y_{0}\in L^{2}\left(  0,L\right)  $, we have
\begin{equation}
\left\Vert S\left(  t\right)  y_{0}\right\Vert _{H_{0}^{1}(0,L)}\leq\frac
{C}{\sqrt{t}}\left\Vert y_{0}\right\Vert _{L^{2}\left(  0,L\right)
},\quad \forall t\in(0,T].\label{T}%
\end{equation}
Consequently, the $C_{0}$-semigroup $\left\{  S\left(  t\right)  \right\}
_{t\geq0}$ generated by $A$ is compact.
\end{proposition}

\begin{proof}
Let $T>0$ be fixed. For every $t\in\left(  0,T\right]$ and for every
$y_0\in L^2(0,L)$,  by Lemma \ref{LE1}, the estimate
\begin{equation}
\left\Vert S(\cdot)y_{0}\right\Vert _{L^{2}\left(  0,\frac{t}{2};H_{0}%
^{1}(0,L)\right)  }\leq\left(  \frac{2t+L}{3}\right)  ^{\frac{1}{2}}\left\Vert
y_{0}\right\Vert _{L^{2}\left(  0,L\right)  }\label{estimate2}%
\end{equation}
holds. Then, arguing by contradiction, we get the existence of $\tau\in\left(
0,t/2\right]  $ such that
\begin{equation}
\left\Vert S\left(  \tau\right)  y_{0}\right\Vert _{H_{0}^{1}(0,L)}\leq\left(
\frac{2t+L}{3}\right)  ^{\frac{1}{2}}\sqrt{\frac{2}{t}}\left\Vert
y_{0}\right\Vert _{L^{2}\left(  0,L\right)}  ,  \forall y_0\in L^2(0,L).\label{estimate1}%
\end{equation}
Now it follows from Lemma \ref{LE2} and (\ref{estimate1}) that there exists
$C^{\prime}=C^{\prime}(T)>0$ such that, for every $t\in\left(  0,T\right]  $
and every $y_{0}\in L^{2}\left(  0,L\right)  $,
\begin{align*}
\left\Vert S\left(  t\right)  y_{0}\right\Vert _{H_{0}^{1}(0,L)} &
=\left\Vert S\left(  t-\tau\right)  S\left(  \tau\right)  y_{0}\right\Vert
_{H_{0}^{1}(0,L)}\\
&  \leq C\left\Vert S\left(  \tau\right)  y_{0}\right\Vert _{H_{0}^{1}\left(
0,L\right)  }\\
&  \leq C\left(  \frac{2t+L}{3}\right)  ^{\frac{1}{2}}\sqrt{\frac{2}{t}%
}\left\Vert y_{0}\right\Vert _{L^{2}\left(  0,L\right)  }\\
&  \leq\frac{C^{\prime}}{\sqrt{t}}\left\Vert y_{0}\right\Vert _{L^{2}\left(
0,L\right)  }.
\end{align*}
Thus, for any given $T>0$, (\ref{T}) holds. Since $H^{1}(0,L)$ is compactly
embedded in $L^{2}(0,L)$, we conclude that $S\left(  t\right)  $ is compact.
\end{proof}

Let us now consider the spectral properties of the operator $A$. Firstly, we
give the definition of growth bound and essential growth bound of the
infinitesimal generator of a linear $C_{0}$-semigroup.

\begin{definition}
\textrm{Let $K:D(K)\subset X\rightarrow X$ be the infinitesimal generator of a
linear $C_{0}$-semigroup $\left\{  S_{K}(t)\right\}  _{t\geq0}$ on a Banach
space $X$. We define $\omega_{0}\left(  K\right)  \in\lbrack-\infty,+\infty)$
the \textit{growth bound }of K by
\[
\omega_{0}\left(  K\right)  :=\lim_{t\rightarrow+\infty}\frac{\ln\left(
\left\Vert S_{K}(t)\right\Vert _{\mathcal{L}\left(  X\right)  }\right)  }{t}.
\]
The \textit{essential growth bound }$\omega_{0,ess}\left(  K\right)
\in\left[  -\infty,+\infty\right)  $ of\textbf{\ }K is defined by
\[
\omega_{0,ess}\left(  K\right)  :=\lim_{t\rightarrow+\infty}\frac{\ln\left(
\left\Vert S_{K}(t)\right\Vert _{ess}\right)  }{t},
\]
where $\left\Vert S_{K}(t)\right\Vert _{ess}$ is the essential norm of
$S_{K}(t)$ defined by
\[
\left\Vert S_{K}(t)\right\Vert _{ess}=\kappa\left(  S_{K}(t)B_{X}\left(
0,1\right)  \right)  ,
\]
where $B_{X}\left(  0,1\right)  :=\left\{  x\in X:\left\Vert x\right\Vert
_{X}\leq1\right\}  $ and, for each bounded set $B\subset X$,
\[
\kappa\left(  B\right)  =\inf\left\{  \varepsilon>0:B\text{ can be covered by
a finite number of balls of radius }\leq\varepsilon\right\}
\]
is the Kuratovsky measure of non-compactness.}
\end{definition}

The following result is proved by Webb \cite[Proposition 4.11, p.~166,
Proposition 4.13, p.170]{Webb1985} and by Engel and Nagel \cite[Corollary
2.11, p.~241]{Engel-Nagel}.

\begin{theorem}
\label{TH1}Let $K:D(K)\subset X\rightarrow X$ be the infinitesimal
generator of a linear $C_{0}$-semigroup $\left\{
S_{K}(t)\right\}_{t\geq 0} $ on a Banach space $X.$ Then
\[
\omega_{0}\left(  K\right)  =\max\left(  \omega_{0,ess}\left(  K\right)
,\max_{\lambda\in\sigma\left(  K\right)  \setminus\sigma_{ess}\left(
K\right)  }Re\left(  \lambda\right)  \right)  .
\]
Assume in addition that $\omega_{0,ess}\left(  K\right)  <\omega_{0}\left(
K\right)  .$ Then for each $\gamma\in\left(  \omega_{0,ess}\left(  K\right)
,\omega_{0}\left(  K\right)  \right]  ,$ \newline$\left\{  \lambda\in
\sigma\left(  K\right)  :Re\left(  \lambda\right)  \geq\gamma\right\}
\subset\sigma_{p}(K)$ is nonempty, finite and contains only poles of the
resolvent of $K$.

\end{theorem}

As a consequence of Proposition \ref{PROP2} and Theorem \ref{TH1}, one has the
following lemma.

\begin{lemma}
\label{LE2.4} All the spectrum of the linear operator $A$ are point spectrum,
i.e., $\sigma\left(  A\right)  =\sigma_{p}\left(  A\right)  $ and $\omega
_{0}\left(  A\right)  =\underset{\lambda\in\sigma\left(  A\right)  }{\max
}Re\left(  \lambda\right)  .$ Moreover, for each $\gamma\in\left(
-\infty,\omega_{0}\left(  A\right)  \right]  ,$ $\left\{  \lambda\in
\sigma\left(  A\right)  :Re\left(  \lambda\right)  \geq\gamma\right\}  $ is
nonempty, finite and contains only poles of the resolvent of $A$.
\end{lemma}

 From Lemma \ref{LE3} and Lemma \ref{LE2.4}, one has

\begin{lemma}
\label{Lm2} For every $\lambda\in\sigma\left(  A\right)  $, $\operatorname{Re}%
\left(  \lambda\right)  \leq0$.
\end{lemma}

Let us now prove the following lemma.

\begin{lemma}
\label{Lm3} One has $\sigma_{p}\left(  A\right)  \cap i\mathbb{R}=\{0\}.$
Moreover, the kernel of $A$ is $a(1-\cos x),$ $a\in\mathbb{R}$.
\end{lemma}

\begin{proof}
We have \bigskip$\lambda\in\sigma_{p}\left(  A\right)  \cap i\mathbb{R}$ if
and only if there exists $\varphi\in H^{3}\left(  0,L\right)  \backslash\{0\}$
such that
\begin{equation}
\left\{
\begin{array}
[c]{l}%
\lambda\varphi+\varphi_{x}+\varphi_{xxx}=0,\\
\varphi\left(  0\right)  =\varphi\left(  L\right)  =\varphi_{x}\left(
L\right)  =0.
\end{array}
\right.  \label{phi1}%
\end{equation}
Multiplying equation (\ref{phi1}) by $\overline{\varphi},$ and then
integrating over $[0,L]$, we obtain%
\begin{equation}
\lambda\int_{0}^{L}\varphi\overline{\varphi}dx+\int_{0}^{L}\varphi
_{x}\overline{\varphi}dx+\int_{0}^{L}\varphi_{xxx}\overline{\varphi
}dx=0.\label{phi3}%
\end{equation}
Taking the real part of (\ref{phi3}), we have
\begin{equation}
\int_{0}^{L}\frac{\varphi_{x}\bar{\varphi}+\bar{\varphi}_{x}\varphi}%
{2}\,dx+\int_{0}^{L}\frac{\varphi_{xxx}\bar{\varphi}+\bar{\varphi}%
_{xxx}\varphi}{2}\,dx=0.\label{intepart}%
\end{equation}
Integrating by parts in \eqref{intepart} and using \eqref{phi1}, we
get
\[
\varphi_{x}\left(  0\right)  =0.
\]
Hence, $\lambda\in\sigma_{p}\left(  A\right)  \cap i\mathbb{R}$ if and only if
there exists $\varphi\in H^{3}\left(  0,L\right)  \backslash\{0\}$ such that
\[
\left\{
\begin{array}
[c]{l}%
\lambda\varphi+\varphi_{x}+\varphi_{xxx}=0,\\
\varphi\left(  0\right)  =\varphi\left(  L\right)  =\varphi_{x}\left(
0\right)  =\varphi_{x}\left(  L\right)  =0,
\end{array}
\right.
\]
and the result of this lemma follows directly from the proof of Rosier
\cite[Lemma 3.5]{Rosier1997}.
\end{proof}

Combining Lemma \ref{LE2.4}, Lemma \ref{Lm2} and Lemma \ref{Lm3}, we obtain
the following corollary.

\begin{corollary}
\label{Corollary on spectrum}$0\in\sigma\left(  A\right)  =\sigma_{p}\left(
A\right)  $ and the other eigenvalues of $A$ have negative real parts which
are bounded away from $0$.

\end{corollary}

\section{Existence and smoothness of the center manifold\newline}

\label{sec3}

This section is devoted to show the existence and smoothness of the center
manifold for system (\ref{0}) with $L=2\pi$ by applying the results given in
\cite{Minh-Wu2004}. We would like to mention that the linear operator $A$ in
our system \eqref{0} with $L=2\pi$ does not satisfy the resolvent estimates
required in \cite{Haragus-Iooss2011}. In particular, $A$ does not generate an
analytic semigroup, but a $C_{0}$-semigroup with a Gevrey property. We refer
to \cite{Chu-Coron-Shang} and \cite{Sun} for this result. Hence, we cannot
apply the results given in \cite{Haragus-Iooss2011} to show the existence and
smoothness of the center manifold.

In order to apply the results given in \cite{Minh-Wu2004}, we need to show
that the nonlinear perturbation has a small global Lipschitz constant. To that
end, we modify the nonlinear part of the original system (\ref{0}) by using
some smooth cut-off mapping, and consider the following equation%

\begin{equation}
\left\{
\begin{array}
[c]{l}%
y_{t}+y_{x}+y_{xxx}+\Phi_{\varepsilon}(\left\Vert y\right\Vert _{L^{2}%
(0,L)})yy_{x}=0,\\
y(t,0)=y(t,L)=0,\\
y_{x}(t,L)=0,\\
y(0,x)=y_{0}(x)\in L^{2}(0,L).
\end{array}
\right.  \label{new system}%
\end{equation}
Here $\varepsilon>0$ is small enough,  and
$\Phi_{\varepsilon}:\left[
0,+\infty\right)  \rightarrow\left[  0,1\right]  $ is defined by%
\[
\Phi_{\varepsilon}\left(  x\right)  =\Phi\left(
\frac{x}{\varepsilon}\right) ,\,\forall x\in\lbrack0,+\infty),
\]
where $\Phi\in C^{\infty}\left(  \left[  0,+\infty\right)  ;\left[
0,1\right]  \right)  $ satisfies
\[
\Phi(x)=\left\{
\begin{array}
[c]{l}%
1,\text{ when }x\in [  0,\displaystyle\frac{1}{2}]  ,\\
\\
0,\text{ when }x\in\left[  1,+\infty\right)  ,
\end{array}
\right.
\]
and
\[
\Phi^{\prime}\leq0.
\]
It can be readily checked that
\begin{align}
&\Phi_{\varepsilon}(x)=1,\text{ when }x\in [0,\displaystyle\frac{1}{2}] ,
\nonumber\\
&\Phi_{\varepsilon}(x)=0,\text{ when }x\in\left[  \varepsilon,+\infty\right)  .
\label{Phi=0}%
\end{align}
Moreover, there exists some constant $C>0$ such that
\begin{equation}
0\leq-\Phi_{\varepsilon}^{\prime}(x)\leq\frac{C}{\varepsilon},\quad\forall
x\in\left[  0,+\infty\right)  . \label{derivative of Phi}%
\end{equation}
In (\ref{derivative of Phi}) and in the following, $C$ denotes various
positive constants, which may vary from line to line, but do not depend on
$\varepsilon\in\left(  0,1\right]  $ and $y_{0}\in L^{2}(0,L).$

\subsection{Well-posedness of \eqref{new system}}

In this section, we prove the following proposition on the global (in positive
time) existence and uniqueness of the solution to system (\ref{new system}).

\begin{proposition}
\label{Glocal well-posedness}For every $y_{0}\!\in\! L^{2}\left(
0,L\right) $, there exists a unique mild solution \[
y\!\in\!C([0,+\infty);\!L^{2}( 0,L) )\cap L_{loc}^{2}\left(
[0,+\infty );H_{0}^{1}\left(  0,L\right)  \right) \] of (\ref{new
system}).
\end{proposition}

In order to prove this proposition, one first points out that

\begin{lemma}
\label{decreasingL2norm}Let $T>0$. If
\[
y\!\in\!C([0,T];\!L^{2}( 0,L) )\cap L^{2}\left(  0,T;H_{0}^{1}\left(
0,L\right)  \right)
\]
is a mild solution of (\ref{new system}), then
\[
\frac{d}{dt}\left(  \int_{0}^{L}y^{2}\left(  t,x\right)  dx\right)  \leq0.
\]

\end{lemma}

\begin{proof}
[Proof]We multiply $y_{t}+y_{x}+y_{xxx}%
+\Phi_{\varepsilon}(\left\Vert y\right\Vert _{L^{2}(0,L)})yy_{x}=0$ by $y$ and
integrate over $\left[  0,L\right]  .$ Using the boundary conditions in
(\ref{new system}) and integrations by parts, we get%
\[
\frac{1}{2}\frac{d}{dt}\int_{0}^{L}y^{2}dx+\frac{1}{2}y_{x}^{2}\left(
t,0\right)  =0.
\]
The lemma follows.
\end{proof}

By Lemma \ref{decreasingL2norm}, in order to prove Proposition
\ref{Glocal well-posedness}, it is sufficient to prove local (in positive time)
existence and uniqueness of the solution to system (\ref{new system}).

\begin{proposition}
\label{proplocalex} \label{local well-posedness}Let $\varepsilon>0,\eta>0.$
There exists $T>0$ such that for every $y_{0}\in L^{2}\left(  0,L\right)  $
with $\left\Vert y_{0}\right\Vert _{L^{2}\left(  0,L\right)  }\leq\eta,$ there
exists a unique solution $y\in C\left(  \left[  0,T\right]  ;L^{2}\left(
0,L\right)  \right)  \cap L_{{}}^{2}\left(  0,T;H_{0}^{1}\left(  0,L\right)
\right)  $ of (\ref{new system}).
\end{proposition}

\begin{proof}
The case where $\Phi_{\varepsilon}\equiv1$ is proved in
\cite{Menzala-Vasconcellos-Zuazua}. Adapting the proof given in
\cite{Menzala-Vasconcellos-Zuazua}, we get the existence of $T$ together with
the existence and uniqueness of mild solution $y$. We briefly give the proof
since some estimates given in the proof will be used later on.

Using the variation of constants formula, system (\ref{new system})
can be written in the following integral form:
\begin{align}
y\left(  t,\cdot\right)   &  =S\left(  t\right)  y_{0}+\int_{0}^{t}S\left(
t-s\right)  \Phi_{\varepsilon}\left(  \left\Vert y\left(  s,\cdot\right)
\right\Vert _{L^{2}\left(  0,L\right)  }\right)  y\left(  s,\cdot\right)
y_{x}\left(  s,\cdot\right)  ds\nonumber\\
&  :=\left[  \phi\left(  y\right)  \right]  \left(  t\right)  .
\label{integral form}%
\end{align}
We will show that the nonlinear map $\phi$ is a contraction from
$Y_{T}:=C\left(  \left[  0,T\right]  ;L^{2}\left(  0,L\right)  \right)  \cap
L^{2}\left(  0,T;H_{0}^{1}\left(  0,L\right)  \right)  $ into itself when
$T>0$ is small enough.

 Firstly, we prove that $\phi$ maps continuously $Y_{T}$ into itself.
Let us first show that if $y\in Y_{T},$ $\Phi_{\varepsilon}\left(
\left\Vert y\right\Vert _{L^{2}\left(  0,L\right)  }\right)
yy_{x}\in L^{1}\left(  0,T;L^{2}\left(  0,L\right)  \right)  $ and
the map $y\rightarrow\Phi_{\varepsilon}\left(  \left\Vert
y\right\Vert _{L^{2}\left( 0,L\right)  }\right)  yy_{x}$ is
continuous. Indeed, let $y,z\in Y_{T}.$ Applying the triangular
inequality, H\"{o}lder's inequality and Sobolev's embedding
$H_{0}^{1}\left(  0,L\right)  \subset C^{0}([0,L])$ together with
\eqref{derivative of Phi}, we get%
\begin{align}
&  \left\Vert \Phi_{\varepsilon}\left(  \left\Vert y\right\Vert _{L^{2}\left(
0,L\right)  }\right)  yy_{x}-\Phi_{\varepsilon}\left(  \left\Vert z\right\Vert
_{L^{2}\left(  0,L\right)  }\right)  zz_{x}\right\Vert _{L^{1}\left(
0,T;L^{2}\left(  0,L\right)  \right)  }\nonumber\\
  \leq &\left\Vert \left(  yy_{x}-zz_{x}\right)  \right\Vert _{L^{2}\left(
0,T;L^{2}\left(  0,L\right)  \right)  }+\left\Vert \left[  \Phi_{\varepsilon
}\left(  \left\Vert y\right\Vert _{L^{2}\left(  0,L\right)  }\right)
-\Phi_{\varepsilon}\left(  \left\Vert z\right\Vert _{L^{2}\left(  0,L\right)
}\right)  \right]  zz_{x}\right\Vert _{L^{1}\left(  0,T;L^{2}\left(
0,L\right)  \right)  }\nonumber\\
 \leq &\left\Vert \left(  y-z\right)  y_{x}+\left( y_{x}-z_{x}\right)
z\right\Vert _{L^{1}\left(  0,T;L^{2}\left( 0,L\right)  \right)
}+\frac {C}{\varepsilon}\left\Vert \left\Vert y-z\right\Vert
_{L^{2}\left( 0,L\right)  }zz_{x}\right\Vert _{L^{1}\left(
0,T;L^{2}\left(  0,L\right)
\right)  }\nonumber\\
 \leq &\int_{0}^{T}\left\Vert \left(  y-z\right)  y_{x}\right\Vert
_{L^{2}\left(  0,L\right)  }dt+\int_{0}^{T}\left\Vert \left(  y_{x}%
-z_{x}\right)  z\right\Vert _{L^{2}\left(  0,L\right)  }dt\nonumber\\
&  +\frac{C}{\varepsilon}\int_{0}^{T}\left\Vert y-z\right\Vert _{L^{2}\left(
0,L\right)  }\left\Vert zz_{x}\right\Vert _{L^{2}\left(  0,L\right)
}dt\nonumber\\
\leq  & \,\,\,C\int_{0}^{T}\left\Vert y-z\right\Vert
_{L^{\infty}\left( 0,L\right)
}\left\Vert y_{x}\right\Vert _{L^{2}\left(  0,L\right)  }dt+C\int_{0}%
^{T}\left\Vert z\right\Vert _{L^{\infty}\left(  0,L\right)  }\left\Vert
y_{x}-z_{x}\right\Vert _{L^{2}\left(  0,L\right)  }dt\nonumber\\
&  +\frac{C}{\varepsilon}\left\Vert y-z\right\Vert _{L^{\infty}(0,T;L^{2}%
\left(  0,L\right)  )}\int_{0}^{T}\left\Vert z\right\Vert _{L^{\infty}\left(
0,L\right)  }\left\Vert z_{x}\right\Vert _{L^{2}\left(  0,L\right)
}dt\nonumber\\
 \leq & \,\,\,C\left\Vert y-z\right\Vert _{L^{2}(0,T;L^{\infty}\left(
0,L\right) )}\left\Vert y_{x}\right\Vert _{L^{2}\left(
0,T;L^{2}\left(  0,L\right)
\right)  }\nonumber\\
&  +C\left\Vert z\right\Vert _{L^{2}(0,T;L^{\infty}\left(  0,L\right)
)}\left\Vert y_{x}-z_{x}\right\Vert _{L^{2}\left(  0,T;L^{2}\left(
0,L\right)  \right)  }\nonumber\\
&  +\frac{C}{\varepsilon}\left\Vert y-z\right\Vert _{L^{\infty}(0,T;L^{2
}\left(  0,L\right)  )}\left\Vert z\right\Vert _{L^{2}(0,T;L^{\infty}\left(
0,L\right)  )}\left\Vert z_{x}\right\Vert _{L^{2}\left(  0,T;L^{2}\left(
0,L\right)  \right)  }. \label{Phiy-Phiz}%
\end{align}
By the classical Gagliardo-Nirenberg inequality, we have%
\begin{equation}
\label{15.1*}\left\Vert u\right\Vert _{L^{\infty}(0,L)}\leq C\left\Vert
u\right\Vert _{L^{2}(0,L)}^{\frac{1}{2}}\left\Vert u_{x}\right\Vert
_{L^{2}(0,L)}^{\frac{1}{2}},\quad \forall u\in H_{0}^{1}\left(  0,L\right)  .
\end{equation}
Hence,%
\begin{align*}
\int_{0}^{T}\left\Vert u\right\Vert _{L^{\infty}\left(  0,L\right)  }^{2}dt
&  \leq C\int_{0}^{T}\left\Vert u\right\Vert _{L^{2}\left(  0,L\right)
}\left\Vert u_{x}\right\Vert _{L^{2}\left(  0,L\right)  }dt\\
&  \leq C\left\Vert u\right\Vert _{L^{\infty}\left(  0,T;L^{2}\left(
0,L\right)  \right)  }\int_{0}^{T}\left\Vert u_{x}\right\Vert _{L^{2}\left(
0,L\right)  }dt\\
&  \leq C\left\Vert u\right\Vert _{L^{\infty}\left(  0,T;L^{2}\left(
0,L\right)  \right)  }T^{\frac{1}{2}}\left\Vert u_{x}\right\Vert
_{L^{2}\left(  0,T;L^{2}\left(  0,L\right)  \right)  }.
\end{align*}
Consequently, we get%
\begin{align*}
\left\Vert u\right\Vert _{L^{2}(0,T;L^{\infty}\left(  0,L\right)  )}  &  \leq
C\left\Vert u\right\Vert _{L^{\infty}\left(  0,T;L^{2}\left(  0,L\right)
\right)  }^{\frac{1}{2}}T^{\frac{1}{4}}\left\Vert u_{x}\right\Vert
_{L^{2}\left(  0,T;L^{2}\left(  0,L\right)  \right)  }^{\frac{1}{2}}\\
&  \leq CT^{\frac{1}{4}}\left\Vert u\right\Vert _{Y_{T}},\quad  \forall u\in Y_{T}.
\end{align*}
Thus, it follows from (\ref{Phiy-Phiz}) that
\begin{align}
&  \left\Vert \Phi_{\varepsilon}\left(  \left\Vert y\right\Vert _{L^{2}\left(
0,L\right)  }\right)  yy_{x}-\Phi_{\varepsilon}\left(  \left\Vert z\right\Vert
_{L^{2}\left(  0,L\right)  }\right)  zz_{x}\right\Vert _{L^{1}\left(
0,T;L^{2}\left(  0,L\right)  \right)  }\nonumber\\
  \leq & \ CT^{\frac{1}{4}}\left\Vert y-z\right\Vert _{Y_{T}}\left\Vert
y_{x}\right\Vert _{L^{2}\left(  0,T;L^{2}\left(  0,L\right)  \right)
} +CT^{\frac{1}{4}}\left\Vert z\right\Vert _{Y_{T}}\left\Vert y_{x}%
-z_{x}\right\Vert _{L^{2}\left(  0,T;L^{2}\left(  0,L\right)  \right)}\nonumber\\
&  +\frac{C}{\varepsilon}\left\Vert y-z\right\Vert _{L^{\infty}\left(
0,T;L^{2}\left(  0,L\right)  \right)  }T^{\frac{1}{4}}\left\Vert z\right\Vert
_{Y_{T}}\left\Vert z_{x}\right\Vert _{L^{2}\left(  0,T;L^{2}\left(
0,L\right)  \right)  }\nonumber\\
  \leq & \left\Vert y-z\right\Vert _{Y_{T}}T^{\frac{1}{4}}C\left(  \left\Vert
y\right\Vert _{Y_{T}}+\left\Vert z\right\Vert _{Y_{T}}+\frac{1}{\varepsilon
}\left\Vert z\right\Vert _{Y_{T}}^{2}\right)  , \label{Phiy-Phiz-final}%
\end{align}
which implies that $\Phi_{\varepsilon}\left( \left\Vert y\right\Vert
_{L^{2}\left(  0,L\right)  }\right) yy_{x}\in L^{1}\left(
0,T;L^{2}\left( 0,L\right)  \right)  $ and that the map
\[
y\rightarrow\Phi_{\varepsilon}\left(  \left\Vert y\right\Vert _{L^{2}\left(
0,L\right)  }\right)  yy_{x}%
\]
is continuous from $Y_{T}$ to $L^{1}\left(  0,T;L^{2}\left(  0,L\right)
\right)  $.

By Proposition 4.1 in \cite{Rosier1997}, we obtain that
\[
\int_{0}^{t}S\left(  t-s\right)  \Phi_{\varepsilon}\left(  \left\Vert y\left(
s,\cdot\right)  \right\Vert _{L^{2}\left(  0,L\right)  }\right)  y\left(
s,\cdot\right)  y_{x}\left(  s,\cdot\right)  ds
\]
lies in $Y_{T},$ and the map
\[
\Phi_{\varepsilon}\left(  \left\Vert y\right\Vert _{L^{2}\left(  0,L\right)
}\right)  yy_{x}\rightarrow\int_{0}^{t}S\left(  t-s\right)  \Phi_{\varepsilon
}\left(  \left\Vert y\left(  s,\cdot\right)  \right\Vert _{L^{2}\left(
0,L\right)  }\right)  y\left(  s,\cdot\right)  y_{x}\left(  s,\cdot\right)
ds
\]
is continuous. This fact, together with the continuity of the map
$y\rightarrow\Phi_{\varepsilon}\left(  \left\Vert y\right\Vert
_{L^{2}\left( 0,L\right)  }\right)  yy_{x}$ from $Y_{T}$ to
$L^{1}\left(  0,T;L^{2}\left( 0,L\right)  \right)  $ and $S\left(
t\right)  y_{0}\in Y_{T}$ (thanks to Lemma \ref{LE3} and Lemma
\ref{LE1}), leads to the conclusion that $\phi$ maps continuously
$Y_{T}$ into itself.

Let us now prove that $\phi$ is a contraction in a suitable ball
 $B_R$ of $Y_{T}$ when $T>0$ is small enough.
Obviously,
\[
\phi\left(  y\right)  -\phi\left(  z\right)  =\int_{0}^{t}S\left(  t-s\right)
\left[  \Phi_{\varepsilon}\left(  \left\Vert y\right\Vert _{L^{2}\left(
0,L\right)  }\right)  yy_{x}\left(  s\right)  -\Phi_{\varepsilon}\left(
\left\Vert z\right\Vert _{L^{2}\left(  0,L\right)  }\right)  zz_{x}\left(
s\right)  \right]  ds.
\]
In view of the proof of Proposition 4.1 in \cite{Rosier1997} and
(\ref{Phiy-Phiz-final}), we deduce that
\begin{align}
&  \left\Vert \phi\left(  y\right)  -\phi\left(  z\right)  \right\Vert
_{Y_{T}}\nonumber\\
 \leq & \left(  1+\left(  \frac{T+2L}{3}\right)  ^{\frac{1}{2}}\right)
\left\Vert \Phi_{\varepsilon}\left(  \left\Vert y\right\Vert _{L^{2}\left(
0,L\right)  }\right)  yy_{x}-\Phi_{\varepsilon}\left(  \left\Vert z\right\Vert
_{L^{2}\left(  0,L\right)  }\right)  zz_{x}\right\Vert _{L^{1}\left(
0,T;L^{2}\left(  0,L\right)  \right)  }\nonumber\\
  \leq &  C\left(  1+\sqrt{T}\right)  \left\Vert \Phi_{\varepsilon}\left(
\left\Vert y\right\Vert _{L^{2}\left(  0,L\right)  }\right)  yy_{x}%
-\Phi_{\varepsilon}\left(  \left\Vert z\right\Vert _{L^{2}\left(  0,L\right)
}\right)  zz_{x}\right\Vert _{L^{1}\left(  0,T;L^{2}\left(  0,L\right)
\right)  }\nonumber\\
 \leq  & C\left(  1+\sqrt{T}\right)  \left\Vert y-z\right\Vert _{Y_{T}}%
T^{\frac{1}{4}}\left(  \left\Vert y\right\Vert _{Y_{T}}+\left\Vert
z\right\Vert _{Y_{T}}+\frac{1}{\varepsilon}\left\Vert z\right\Vert _{Y_{T}%
}^{2}\right)  ,\label{phi(y)-phi(z)}%
\end{align}
which shows that $\phi$ is a contraction in the ball $B_{R}$ of $Y_{T}$ if%
\begin{equation}
C\left(  1+\sqrt{T}\right)  T^{\frac{1}{4}}\left(  2R+\frac{1}{\varepsilon
}R^{2}\right)  <1. \label{T-R}%
\end{equation}
Therefore, the proof will be complete if we could show that for a
suitable choice of $R$ and $T$ satisfying (\ref{T-R}), the map
$\phi$ sends $B_{R}$ into itself.

It can be deduced from the definition of $\phi\left(  y\right)  $
given in (\ref{integral form}), Lemma \ref{LE3}, Lemma \ref{LE1} and
(\ref{phi(y)-phi(z)}) with $z=0$ that there exists $\bar C>0$
independent of $\varepsilon\in(0,1]$, $y_{0}\in L^{2}(0,L)$ and
$T>0$,  such that
\begin{align*}
\left\Vert \phi\left(  y\right)  \right\Vert _{Y_{T}}  &  \leq\left(
1+\left(  \frac{4T+L}{3}\right)  ^{\frac{1}{2}}\right)  \left\Vert
y_{0}\right\Vert _{L^{2}\left(  0,L\right)  }+\left\Vert y\right\Vert _{Y_{T}%
}^{2}T^{\frac{1}{4}}C\left(  1+\sqrt{T}\right) \\
&  \leq\left(  1+\left(  \frac{4T+L}{3}\right)  ^{\frac{1}{2}}\right)
\left\Vert y_{0}\right\Vert _{L^{2}\left(  0,L\right)  }+R^{2}T^{\frac{1}{4}%
}C\left(  1+\sqrt{T}\right) \\
&  \leq\bar C\left(  1+\sqrt{T}\right)  \left(  \left\Vert y_{0}\right\Vert
_{L^{2}\left(  0,L\right)  }+R^{2}T^{\frac{1}{4}}\right)  ,\quad \forall y\in B_{R}.
\end{align*}
Now let $\left\Vert y_{0}\right\Vert _{L^{2}\left( 0,L\right)} \leq
\eta$, and set $R:=2\eta\bar C. $
Then%
\begin{equation}
\label{phinormest}\left\Vert \phi\left(  y\right)  \right\Vert _{Y_{T}}%
\leq\eta\bar C\left(  1+\sqrt{T}\right)  \left(  1+4\bar C^{2} \eta
T^{\frac{1}{4}}\right)  ,\quad \forall y\in B_{R}.
\end{equation}
It is clear that we can choose $T>0$ sufficiently small such that
\[
\left(  1+\sqrt{T}\right)  \left(  1+4 \eta\bar C^{2} T^{\frac{1}{4}}\right)
\leq2,
\]
which, together with \eqref{phinormest} implies that $\phi$ maps $B_{R}$ into
itself. Moreover, decreasing $T$ if necessary allows us to guarantee
(\ref{T-R}) as well. The proof of Proposition \ref{proplocalex} is complete.
\end{proof}

\begin{proposition}
\label{proplocal-weel-posed} There exists $C>0$ such that for every
$\varepsilon>0,$ for every $y_{0}\in L^{2}\left(  0,L\right)  $ and
for every $T>0,$ the unique solution of (\ref{new system}) satisfies
\begin{equation} \left\Vert y\right\Vert _{L^{2}\left(
0,T;H_{0}^{1}\left(  0,L\right) \right)
}^{2}\leq\frac{8T+2L}{3}\left\Vert y_{0}\right\Vert _{L^{2}\left(
0,L\right)  }^{2}+CT\left\Vert y_{0}\right\Vert
_{L^{2}\left(
0,L\right)  }^{4}. \label{Prop-local well-posedness-inequality}%
\end{equation}

\end{proposition}

\begin{proof}
Proceeding as in \cite{Rosier1997}, we multiply the first equation
in (\ref{new system}) by $xy$ and integrate over \thinspace$\left(
0,L\right) \times\left(  0,T\right)  $. Then,  by
 Lemma~\ref{decreasingL2norm}, we
obtain%
\begin{align}
&  \int_{0}^{T}\int_{0}^{L}y_{x}^{2}dxdt+\frac{1}{3}\int_{0}^{L}xy^{2}\left(
x,T\right)  dx\nonumber\\
  = &\frac{1}{3}\int_{0}^{T}\int_{0}^{L}y^{2}dxdt+\frac{1}{3}\int_{0}^{L}%
xy_{0}^{2}dx-\frac{2}{3}\int_{0}^{T}\Phi_{\varepsilon}(\left\Vert y\right\Vert
_{L^{2}(0,L)})\int_{0}^{L}xy^{2}y_{x}dxdt\nonumber\\
 \leq & \frac{T+L}{3}\left\Vert y_{0}\right\Vert _{L^{2}(0,L)}^{2}+\frac{2}%
{3}\int_{0}^{T}\Phi_{\varepsilon}(\left\Vert y\right\Vert _{L^{2}%
(0,L)})\left\vert \int_{0}^{L}xy^{2}y_{x}dx\right\vert dt. \label{T-L-y_x^2}%
\end{align}
Since
\begin{equation*}
\int_{0}^{L}xy^{2}y_{x}dx=-\frac{1}{3}\int_{0}^{L}y^{3}dx, \label{xy^2y_x}%
\end{equation*}
it follows from (\ref{T-L-y_x^2}) that
\begin{align*}
&  \int_{0}^{T}\int_{0}^{L}y_{x}^{2}dxdt+\frac{1}{3}\int_{0}^{L}xy^{2}\left(
x,T\right)  dx\\
  \leq & \frac{T+L}{3}\left\Vert y_{0}\right\Vert _{L^{2}(0,L)}^{2}+\frac{2}%
{9}\int_{0}^{T}\Phi_{\varepsilon}(\left\Vert y\right\Vert _{L^{2}(0,L)}%
)\int_{0}^{L}\left\vert y\right\vert ^{3}dxdt\\
  \leq & \frac{T+L}{3}\left\Vert y_{0}\right\Vert _{L^{2}(0,L)}^{2}+\frac{2}%
{9}\int_{0}^{T}\int_{0}^{L}\left\vert y\right\vert ^{3}dxdt.
\end{align*}
Hence,%
\begin{equation}
\left\Vert y\right\Vert _{L^{2}\left(  0,T;H_{0}^{1}\left(  0,L\right)
\right)  }^{2}\leq\frac{4T+L}{3}\left\Vert y_{0}\right\Vert _{L^{2}(0,L)}%
^{2}+\frac{2}{9}\int_{0}^{T}\int_{0}^{L}\left\vert y\right\vert ^{3}dxdt.
\label{y_L^2(0,T,H)^2}%
\end{equation}
 Furthermore, by Lemma~\ref{decreasingL2norm}, the continuous Sobolev
embedding $H_{0}^{1}(0,L)\subset C^{0}([0,L])$, Poincar\'{e}
inequality and H\"{o}lder's inequality, we have
\begin{align*}
\int_{0}^{T}\int_{0}^{L}\left\vert y\right\vert ^{3}dxdt  &  \leq C\int
_{0}^{T}\left\Vert y\right\Vert _{H_{0}^{1}\left(  0,L\right)  }\left(
\int_{0}^{L}\left\vert y\right\vert ^{2}dx\right)  dt\\
&  \leq C\left\Vert y_{0}\right\Vert _{L^{2}(0,L)}^{2}\int_{0}^{T}\left\Vert
y\right\Vert _{H_{0}^{1}\left(  0,L\right)  }dt\\
&  \leq C\left\Vert y_{0}\right\Vert _{L^{2}(0,L)}^{2}\sqrt{T}\left(  \int
_{0}^{T}\left\Vert y\right\Vert _{H_{0}^{1}\left(  0,L\right)  }^{2}dt\right)
^{\frac{1}{2}}\\
&  =C\sqrt{T}\left\Vert y_{0}\right\Vert _{L^{2}(0,L)}^{2}\left\Vert
y\right\Vert _{L^{2}(0,T;H_{0}^{1}\left(  0,L\right)  )}.
\end{align*}
Now, using the above inequality in (\ref{y_L^2(0,T,H)^2}) we have%
\begin{align*}
&  \left\Vert y\right\Vert _{L^{2}\left(  0,T;H_{0}^{1}\left(
0,L\right)
\right)  }^{2}\\
&  \leq\frac{4T+L}{3}\left\Vert y_{0}\right\Vert _{L^{2}(0,L)}^{2}+C%
\sqrt{T}\left\Vert y_{0}\right\Vert _{L^{2}(0,L)}^{2}\left\Vert
y\right\Vert _{L^{2}(0,T;H_{0}^{1}\left(  0,L\right)  )}\\
&  \leq\frac{4T+L}{3}\left\Vert y_{0}\right\Vert
_{L^{2}(0,L)}^{2}+
CT\left\Vert y_{0}\right\Vert _{L^{2}(0,L)}^{4}+\frac{1}%
{2}\left\Vert y\right\Vert _{L^{2}(0,T;H_{0}^{1}\left(  0,L\right)
)}^{2}.
\end{align*}
Therefore, we get%
\[
\left\Vert y\right\Vert _{L^{2}\left(  0,T;H_{0}^{1}\left(
0,L\right)
\right)  }^{2}\leq\frac{8T+2L}{3}\left\Vert y_{0}\right\Vert _{L^{2}(0,L)}%
^{2}+CT\left\Vert y_{0}\right\Vert
_{L^{2}(0,L)}^{4}.
\]
This concludes the proof of Proposition~\ref{proplocal-weel-posed}.
\end{proof}

\begin{remark}
\label{R1}According to Proposition~\ref{proplocal-weel-posed}, we
have, for
every $\tau\in[0,T]$,%
\[
\left\Vert y\right\Vert _{L^{2}\left(  \tau,T;H_{0}^{1}\left(
0,L\right) \right)  }^{2}\leq\frac{8\left(  T-\tau\right)
+2L}{3}\left\Vert y\left( \tau,\cdot\right)  \right\Vert
_{L^{2}(0,L)}^{2}+C\left( T-\tau\right) \left\Vert
y\left(  \tau,\cdot\right)  \right\Vert _{L^{2}(0,L)}^{4}.
\]
It follows that, if $\tau\in\left[  0,T\right]  $ is such that
$\left\Vert y\left(  \tau,\cdot\right)  \right\Vert
_{L^{2}(0,L)}=\varepsilon$, then
\begin{align*}
\left\Vert y\right\Vert _{L^{2}\left(  \tau,T;H_{0}^{1}\left(
0,L\right) \right)  }^{2}  &  \leq\frac{8\left(  T-\tau\right)
+2L}{3}\varepsilon
^{2}+C\left(  T-\tau\right)\varepsilon^{4}\\
&
\leq\frac{8T+2L}{3}\varepsilon^{2}+C T\varepsilon^{4}.
\end{align*}
\end{remark}

\begin{lemma}
\label{y_x-L^2}Let $T>0.$ There exist $\eta>0$ and $C>0,$ such that, for every
$\varepsilon\in\left(  0,1\right]  $ and for every $y_{0}\in L^{2}\left(
0,L\right)  $ with $\left\Vert y_{0}\right\Vert _{L^{2}\left(  0,L\right)
}\leq\eta,$ there exists a unique mild solution $y:\left[  0,T\right]
\times\left[  0,L\right]  \rightarrow\mathbb{R}$ of (\ref{new system}) which satisfies
\[
\left\Vert y\left(  t,\cdot\right)  \right\Vert _{H_{0}^{1}\left(  0,L\right)
}\leq\frac{C}{\sqrt{t}}\left\Vert y_{0}\right\Vert _{L^{2}\left(  0,L\right)
},\quad \forall t\in\left(  0,T\right]  .
\]

\end{lemma}

\begin{proof}
 From Proposition~\ref{PROP2} and (\ref{integral form}), we deduce that
\begin{align}
\left\Vert y\left(  t,\cdot\right)  \right\Vert _{H_{0}^{1}\left(  0,L\right)
}  &  \leq\left\Vert S\left(  t\right)  y_{0}\right\Vert _{H_{0}^{1}\left(
0,L\right)  }+\int_{0}^{t}\left\Vert S\left(  t-s\right)  \Phi_{\varepsilon
}\left(  \left\Vert y\left(  s,\cdot\right)  \right\Vert _{L^{2}\left(
0,L\right)  }\right)  y\left(  s,\cdot\right)  y_{x}\left(  s,\cdot\right)
\right\Vert _{H_{0}^{1}\left(  0,L\right)  }ds\nonumber\\
&  \leq\frac{C}{\sqrt{t}}\left\Vert y_{0}\right\Vert _{L^{2}\left(
0,L\right)  }+\int_{0}^{t}\frac{C}{\sqrt{t-s}}\left\Vert y\left(
s,\cdot\right)  y_{x}\left(  s,\cdot\right)  \right\Vert
_{L^{2}\left( 0,L\right)  }ds. \label{Lemma3.2-1}
\end{align}
As a consequence of Lemma~\ref{decreasingL2norm} and \eqref{15.1*}, we have
\begin{align}
\left\Vert y\left(  s,\cdot\right)  y_{x}\left(  s,\cdot\right)  \right\Vert
_{L^{2}\left(  0,L\right)  }  &  \leq\left\Vert y\left(  s,\cdot\right)
\right\Vert _{L^{\infty}\left(  0,L\right)  }\left\Vert y_{x}\left(
s,\cdot\right)  \right\Vert _{L^{2}\left(  0,L\right)  }\nonumber\\
&  \leq C \left\Vert y\left(  s,\cdot\right)  \right\Vert _{L^{2}\left(
0,L\right)  }^{\frac{1}{2}}\left\Vert y_{x}\left(  s,\cdot\right)  \right\Vert
_{L^{2}\left(  0,L\right)  }^{\frac{3}{2}}\nonumber\\
&  \leq C \left\Vert y_{0}\right\Vert _{L^{2}\left(  0,L\right)
}^{\frac {1}{2}}\left\Vert y\left(  s,\cdot\right)  \right\Vert
_{H_{0}^{1}\left( 0,L\right)  }^{\frac{3}{2}}.\label{Lemma3.2-2}
\end{align}
Substituting (\ref{Lemma3.2-2}) into (\ref{Lemma3.2-1}), we obtain
\[
\left\Vert y\left(  t,\cdot\right)  \right\Vert _{H_{0}^{1}\left(  0,L\right)
}\leq\frac{C}{\sqrt{t}}\left\Vert y_{0}\right\Vert _{L^{2}\left(  0,L\right)
}+\left\Vert y_{0}\right\Vert _{L^{2}\left(  0,L\right)  }^{\frac{1}{2}}%
\int_{0}^{t}\frac{C}{\sqrt{t-s}}\left\Vert y\left(  s,\cdot\right)
\right\Vert _{H_{0}^{1}\left(  0,L\right)  }^{\frac{3}{2}}ds,
\]
i.e.
\begin{align}
&  \sqrt{t}\left\Vert y\left(  t,\cdot\right)  \right\Vert _{H_{0}^{1}\left(
0,L\right)  }\nonumber\\
 \leq & \  C\left\Vert y_{0}\right\Vert _{L^{2}\left(  0,L\right)
}+\left\Vert y_{0}\right\Vert _{L^{2}\left(  0,L\right)
}^{\frac{1}{2}}\sqrt{t}\int
_{0}^{t}\frac{C}{s^{\frac{3}{4}}\sqrt{t-s}}\left(
\sqrt{s}\left\Vert y\left( s,\cdot\right)  \right\Vert
_{H_{0}^{1}\left(  0,L\right)  }\right)
^{\frac{3}{2}}ds. \label{contradiction}%
\end{align}
\newline Let $\overline{C}>C$. We claim that there exists $\eta>0$ (small
enough) such that, for every $\varepsilon\in\left(  0,1\right]  $
and for every $y_{0}\in L^{2}\left(  0,L\right)  $ such that
$\left\Vert y_{0}\right\Vert _{L^{2}\left(  0,L\right) }\leq\eta,$
we have
\begin{equation}
\xi\left(  t\right)  <\overline{C}\left\Vert y_{0}\right\Vert _{L^{2}\left(
0,L\right)  },\quad \forall t\in\left(  0,T\right]  , \label{claim}%
\end{equation}
where $\xi\left(  t\right)  :=\sqrt{t}\left\Vert y\left(  t,\cdot\right)
\right\Vert _{H_{0}^{1}\left(  0,L\right)  }$. Let us argue by contradiction.
Suppose that (\ref{claim}) is not valid. Then there exists $\tau\in\left(
0,T\right]  $ such that
\begin{equation}
\xi\left(  \tau\right)  = \overline{C}\left\Vert y_{0}\right\Vert
_{L^{2}\left(  0,L\right)  }\text{ and }\xi\left(  t\right)  <\overline
{C}\left\Vert y_{0}\right\Vert _{L^{2}\left(  0,L\right)  },\quad \forall
t\in\left(  0,\tau\right)  . \label{contradiction 1}%
\end{equation}
Thus by (\ref{contradiction}), we have
\begin{align*}
\xi\left(  \tau\right)   &  \leq C\left\Vert y_{0}\right\Vert _{L^{2}\left(
0,L\right)  }+\left\Vert y_{0}\right\Vert _{L^{2}\left(  0,L\right)  }%
^{\frac{1}{2}}\sqrt{\tau}\int_{0}^{\tau}\frac{C}{s^{\frac{3}{4}}\sqrt{\tau-s}%
}\left(  \overline{C}\left\Vert y_{0}\right\Vert _{L^{2}\left(  0,L\right)
}\right)  ^{\frac{3}{2}}ds\\
&  =C\left\Vert y_{0}\right\Vert _{L^{2}\left(  0,L\right)  }+\left\Vert
y_{0}\right\Vert _{L^{2}\left(  0,L\right)  }^{2}\sqrt{\tau}C\left(
\overline{C}\right)  ^{\frac{3}{2}}\int_{0}^{\tau}\frac{1}{s^{\frac{3}{4}%
}\sqrt{\tau-s}}ds\\
&  =\left\Vert y_{0}\right\Vert _{L^{2}\left(  0,L\right)  }\left(
C+\left\Vert y_{0}\right\Vert _{L^{2}\left(  0,L\right)  }\sqrt{\tau}C\left(
\overline{C}\right)  ^{\frac{3}{2}}\int_{0}^{\tau}\frac{1}{s^{\frac{3}{4}%
}\sqrt{\tau-s}}ds\right)  .
\end{align*}
It can be readily checked that if $\left\Vert y_{0}\right\Vert _{L^{2}\left(
0,L\right)  }$ is small enough, we get $\xi\left(  \tau\right)  <\overline
{C}\left\Vert y_{0}\right\Vert _{L^{2}\left(  0,L\right)  },$ which leads to a
contradiction with (\ref{contradiction 1}). This concludes the proof of
Lemma~\ref{y_x-L^2}.
\end{proof}

\subsection{\bigskip Properties of the semigroup generated by
(\ref{new system})}

Let
\[
\mathcal{S}(t):L^{2}\left(  0,L\right)  \rightarrow L^{2}\left(  0,L\right)
,\quad t\geq0
\]
be the semigroup on $L^{2}\left(  0,L\right)  $ defined by
\[
\mathcal{S}(t)(y_{0}):=y\left(  t,x\right)  ,
\]
where $y\left(  t,x\right)  $ is the unique solution of (\ref{new
system}) with respect to the initial value $y_{0}\in L^{2}\left(
0,L\right) $. Let $T>0$. Then, for every $t\in\left[ 0,T\right]  $,
$\mathcal{S}(t)$ can be decomposed as
\begin{equation*}
\mathcal{S}(t)=S(t)+R(t),
\end{equation*}
or equivalently,
\[
y(t,x)=z\left(  t,x\right)  +\alpha\left(  t,x\right)  ,
\]
where, as above, for every $y_{0}\in L^{2}\left(  0,L\right)  $, $z\left(
t,\cdot\right)  :=S\left(  t\right)  y_{0}$ is the unique solution of
\[
\left\{
\begin{array}
[c]{l}%
z_{t}+z_{x}+z_{xxx}=0,\\
z\left(  t,0\right)  =z\left(  t,L\right)=0,\\
z_{x}\left(  t,L\right)  =0\\
z\left(  0,x\right)  =y_{0}%
\end{array}
\right.
\]
and $\alpha\left(  t,\cdot\right)  :=R(t)y_{0}$ is the unique solution of
\[
\left\{
\begin{array}
[c]{l}%
\alpha_{t}+\alpha_{x}+\alpha_{xxx}+\Phi_{\varepsilon}\left(  \left\Vert
z+\alpha\right\Vert _{L^{2}\left(  0,L\right)  }\right)  \left(  z_{x}%
\alpha+\alpha_{x}z+z_{x}z+\alpha_{x}\alpha\right)  =0,\\
\alpha\left(  t,0\right)  =\alpha\left(  t,L\right)=0,\\
\alpha_{x}\left(
t,L\right)  =0,\\
\alpha\left(  0,x\right)  =0.
\end{array}
\right.
\]
Let%

\[
M:=\left\{  \alpha\varphi:\alpha\in\mathbb{R}\right\}  ,
\]
where%
\begin{equation}
\varphi\left(  x\right)  =\frac{1}{\sqrt{3\pi}}\left(  1-\cos x\right)  .
\label{phi}%
\end{equation}
Let us recall that, by Lemma~\ref{Lm3}, $\varphi\left(  x\right)  $ is an
eigenfunction of the linear operator $A$ for the linearized system
(\ref{linearized}) corresponding to the eigenvalue $0$ and $M$ is the
eigenspace corresponding to this eigenvalue. Then we can do the following
decomposition of $X=L^{2}\left(  0,L\right)  $:
\[
X=M\oplus M^{\bot}.
\]
The projection $P:X\rightarrow M$ is given by
\[
Py(t,x)=p(t)\varphi(x),
\]
where
\begin{equation}
p(t):=\int_{0}^{L}y(t,x)\varphi(x)dx, \label{projection-on-M}%
\end{equation}
and the projection $Q:X\rightarrow M^{\bot}$ is given by $I-P.$

It is clear that $S(t)$ leaves $M$ and $M^{\bot}$ invariant and
$S(t)$ commutes with $P$ and $Q.$ Denote by $S_{1}(t):M\rightarrow
M$ and $S_{2}(t):M^{\bot}\rightarrow M^{\bot}$ the restriction of
$S(t)$ on $M$ and $M^{\bot}$ respectively. Then $S_{1}(t)=Id.$
Moreover, by  Corollary \ref{Corollary on spectrum}, there exist
$N\geq1$ and $\omega>0$ such that $\left\Vert S_{2}(t)\right\Vert
\leq Ne^{-\omega t}, \forall t\geq0.$

\subsubsection{\bigskip Global Lipschitzianity of the map $R\left(  t\right)  :L^{2}\left(
0,L\right)  \rightarrow
L^{2}\left(  0,L\right) $}

The aim of this part is to prove and estimate the global
Lipschitzianity of the map $R\left(  t\right)  :L^{2}\left(
0,L\right)  \rightarrow
L^{2}\left(  0,L\right)  $. To that end, we consider%
\[
\left\{
\begin{array}
[c]{l}%
\alpha_{t}+\alpha_{x}+\alpha_{xxx}+\Phi_{\varepsilon}\left(  \left\Vert
\alpha+z\right\Vert _{L^{2}\left(  0,L\right)  }\right)  \left(  z_{x}%
\alpha+\alpha_{x}z+z_{x}z+\alpha_{x}\alpha\right)  =0,\\
\alpha\left(  t,0\right)  =\alpha\left(  t,L\right)=0,\\
\alpha_{x}\left(
t,L\right)  =0,\\
\alpha\left(  0,x\right)  =0,
\end{array}
\right.
\]
and%
\[
\left\{
\begin{array}
[c]{l}%
\overline{\alpha}_{t}+\overline{\alpha}_{x}+\overline{\alpha}_{xxx}%
+\Phi_{\varepsilon}\left(  \left\Vert \overline{z}+\overline{\alpha
}\right\Vert _{L^{2}\left(  0,L\right)  }\right)  \left(  \overline{z}%
_{x}\overline{\alpha}+\overline{\alpha}_{x}\overline{z}+\overline{z}%
_{x}\overline{z}+\overline{\alpha}_{x}\overline{\alpha}\right)  =0,\\
\overline{\alpha}\left(  t,0\right)  =\overline{\alpha}\left(  t,L\right)=0,\\
\overline{\alpha}_{x}\left(  t,L\right)  =0,\\
\overline{\alpha}\left(  0,x\right)  =0,
\end{array}
\right.
\]
where $z$ is the solution of
\[
\left\{
\begin{array}
[c]{l}%
z_{t}+z_{x}+z_{xxx}=0,\\
z\left(  t,0\right)  =z\left(  t,L\right)  =0,\\
z_{x}\left(  t,L\right)  =0,\\
z\left(  0,x\right)  =y_{0}\in L^{2}\left(  0,L\right)  ,
\end{array}
\right.
\]
and $\overline{z}$ is the solution of
\[
\left\{
\begin{array}
[c]{l}%
\overline{z}_{t}+\overline{z}_{x}+\overline{z}_{xxx}=0,\\
\overline{z}\left(  t,0\right)  =\overline{z}\left(  t,L\right)  =0,\\
\overline
{z}_{x}\left(  t,L\right)  =0,\\
z\left(  0,x\right)  =\overline{y}_{0}\in L^{2}\left(  0,L\right)  .
\end{array}
\right.
\]
Set
\begin{align*}
\Delta & :=\alpha-\overline{\alpha},\quad
y:=\alpha+z,\quad\overline{y}:=\overline
{z}+\overline{\alpha},\\
\Phi_{1}  &  :=\Phi_{\varepsilon}\left(  \left\Vert y\right\Vert
_{L^{2}\left(  0,L\right)  }\right)
,\quad\Phi_{2}:=\Phi_{\varepsilon}\left( \left\Vert
\overline{y}\right\Vert _{L^{2}\left(  0,L\right)  }\right)  .
\end{align*}
Then we obtain%
\begin{equation}
\left\{
\begin{array}
[c]{l}%
\Delta_{t}+\Delta_{x}+\Delta_{xxx}=-\Phi_{1}yy_{x}+\Phi_{2}\overline
{y}\overline{y}_{x}=\Phi_{1}\left[  -\left(  \alpha+z\right)  \Delta
_{x}-\left(  \overline{\alpha}_{x}+z_{x}\right)  \Delta-\overline{\alpha
}\left(  z-\overline{z}\right)  _{x}\right. \\
 \quad \quad\quad\quad\quad\quad\quad \quad\,\left. -\overline{\alpha}_{x}\left( z-\overline{z}\right)
-z_{x}z+\overline{z}_{x}\overline{z}\right] -\left(  \Phi_{1}-\Phi
_{2}\right)  \left(  \overline{z}_{x}\overline{\alpha}+\overline{\alpha}%
_{x}\overline{z}+\overline{z}_{x}\overline{z}+\overline{\alpha}_{x}%
\overline{\alpha}\right)  ,\\
\Delta\left(  t,0\right)  =\Delta\left(  t,L\right) =0,\\
\Delta_{x}\left(
t,L\right)  =0,\\
\Delta\left(  0,x\right)  =0.
\end{array}
\right.  \label{lambda1}%
\end{equation}
Moreover, by the definition of $\Phi_{1},\,\Phi_{2}$ and
(\ref{Phi=0}), we get
\begin{equation}
\Phi_{1}=\Phi_{2}=0,\text{ \ \ \ }\forall\left\Vert y\right\Vert
_{L^{2}\left(  0,L\right)  }\geq\varepsilon,\, \forall \left\Vert \overline{y}\right\Vert _{L^{2}\left(
0,L\right)  }\geq\varepsilon. \label{Phi1=Phi2=0}%
\end{equation}
We first give the following estimate of the $L^{2}$-norm of
$\Delta$.

\begin{lemma}
\label{delta bounded}Let $T>0.$ Then there exists $C>0$ such that
\[
\left\Vert \Delta(t,\cdot)\right\Vert _{L^{2}\left(  0,L\right)  }\leq C,\,
\forall t\in\left[  0,T\right]  ,\, \forall\varepsilon\in\left(  0,1\right]
,\, \forall y_{0}\in L^{2}(0,L), \, \forall\overline{y}_{0} \in L^{2}(0,L) .
\]

\end{lemma}

\begin{proof}
By integrating by parts in
\[
\int_{0}^{L}\Delta\left(  \Delta_{t}+\Delta_{x}+\Delta_{xxx}+\Phi_{1}%
yy_{x}-\Phi_{2}\overline{y}\overline{y}_{x}\right)  dx=0,
\]
we get%
\begin{equation}
\frac{1}{2}\frac{d}{dt}\int_{0}^{L}\Delta^{2}dx+\frac{1}{2}\Delta_{x}%
^{2}\left(  t,0\right)  =-\Phi_{1}\int_{0}^{L}\Delta yy_{x}dx+\Phi_{2}\int
_{0}^{L}\Delta\overline{y}\overline{y}_{x}dx. \label{gamma}%
\end{equation}
Note that $\Delta(t,0)=\Delta(t,L)=0$, by the continuous Sobolev embedding $H_{0}%
^{1}\left(  0,L\right)  \subset C^{0}\left(  \left[  0,L\right]  \right)  $
and Poincar\'{e} inequality, we obtain%
\begin{align*}
\left\vert \int_{0}^{L}\Delta\overline{y}\overline{y}_{x}dx\right\vert  &
\leq\left\Vert \overline{y}\right\Vert _{L^{\infty}\left(  0,L\right)  }%
\int_{0}^{L}\left\vert \Delta\overline{y}_{x}\right\vert dx\\
&  \leq C\left\Vert \overline{y}\right\Vert _{H_{0}^{1}\left(  0,L\right)
}\int_{0}^{L}\left\vert \Delta\overline{y}_{x}\right\vert dx\\
&  \leq C\left\Vert \overline{y}_{x}\right\Vert _{L^{2}\left(  0,L\right)
}\int_{0}^{L}\left\vert \Delta\overline{y}_{x}\right\vert dx.
\end{align*}
In the above inequalities and in the following, $C$, unless otherwise
specified, denotes various positive constants which may vary from line to line
but are independent of $t\in[0,T]$, $\varepsilon\in(0,1]$, $y_{0}\in
L^{2}(0,L)$ and $\overline{y}_{0}\in L^{2}(0,L)$. Thus,%
\[
\left\vert \int_{0}^{L}\Delta\overline{y}\overline{y}_{x}dx\right\vert \leq
C\left\Vert \overline{y}_{x}\right\Vert _{L^{2}\left(  0,L\right)  }%
^{2}\left\Vert \Delta\right\Vert _{L^{2}\left(  0,L\right)  }.
\]
Similarly, we have%
\[
\left\vert \int_{0}^{L}\Delta yy_{x}dx\right\vert \leq C\left\Vert
y_{x}\right\Vert _{L^{2}\left(  0,L\right)  }^{2}\left\Vert \Delta\right\Vert
_{L^{2}\left(  0,L\right)  }.
\]
Hence, it follows from (\ref{gamma}) that
\[
\frac{d}{dt}\int_{0}^{L}\Delta^{2}dx+\Delta_{x}^{2}\left(  t,0\right)  \leq
C\left(  \Phi_{1}\left\Vert y_{x}\right\Vert _{L^{2}\left(  0,L\right)  }%
^{2}+\Phi_{2}\left\Vert \overline{y}_{x}\right\Vert _{L^{2}\left(  0,L\right)
}^{2}\right)  \left\Vert \Delta\right\Vert _{L^{2}\left(  0,L\right)  }.
\]
In particular,%
\[
\frac{d}{dt}\int_{0}^{L}\Delta^{2}dx\leq C\left(  \Phi_{1}\left\Vert
y_{x}\right\Vert _{L^{2}\left(  0,L\right)  }^{2}+\Phi_{2}\left\Vert
\overline{y}_{x}\right\Vert _{L^{2}\left(  0,L\right)  }^{2}\right)
\left\Vert \Delta\right\Vert _{L^{2}\left(  0,L\right)  }.
\]
By Lemma 17 in \cite{Coron-Crepeau} and  Remark \ref{R1}, we get%
\begin{align*}
\int_{0}^{L}\Delta^{2}dx  &  \leq3\left(  \int_{0}^{t}C\left(  \Phi
_{1}\left\Vert y_{x}\right\Vert _{L^{2}\left(  0,L\right)
}^{2}+\Phi
_{2}\left\Vert \overline{y}_{x}\right\Vert _{L^{2}\left(  0,L\right)  }%
^{2}\right)  dt\right)  ^{2}\\
&  \leq3C^{2}\left(  2\left(  \frac{8T+2L}{3}\varepsilon^{2}+C%
T\varepsilon^{4}\right)  \right)  ^{2},\forall t\in\left[
0,T\right].
\end{align*} The result follows$.$
\end{proof}

 For the sake of simplicity, we denote from now on by $L^2(L^2)$ the norm $L^2(0,T;L^2(0,L))$.
\begin{lemma}
\label{estimate of delta}Let $T>0.$ Then there exists $C>0$ such that
\begin{align*}
&  \|\Delta(t,\cdot)\|_{L^2(0,L)}\nonumber\\
\leqslant
&\ \int_0^T\Big[\Phi_1\left(  \left\Vert\overline
{\alpha}_{x}\right\Vert _{L^{2}\left(  0,L\right)  }+\left\Vert z_{x}%
\right\Vert _{L^{2}\left(  0,L\right)  }+\left\Vert \overline{z}%
_{x}\right\Vert _{L^{2}\left(  0,L\right)  }\right)\left\Vert \left(
z-\overline{z}\right)  _{x}\right\Vert _{L^{2}\left(  0,L\right)  }\nonumber\\
&\qquad +\left\vert
\Phi_{1}-\Phi_{2}\right\vert \left\Vert \overline{z}+\overline{\alpha
}\right\Vert _{L^{2}(0,L)}^{\frac{1}{2}}\left\Vert \left(  \overline{z}+\overline{\alpha
}\right)  _{x}\right\Vert _{L^{2}\left(  0,L\right)  }^{\frac{3}{2}}\Big]dt\nonumber\\
&\times \exp\left[  C\left(  1+\left\Vert \sqrt{\Phi_{1}}\alpha
_{x}\right\Vert _{L^{2}(L^{2})}^{2}+\left\Vert \sqrt{\Phi_{1}}\overline
{\alpha}_{x}\right\Vert _{L^{2}(L^{2})}^{2}+\left\Vert \sqrt{\Phi_{1}}%
z_{x}\right\Vert _{L^{2}(L^{2})}^{2}\right)  \right],
\end{align*}
for every $t\in\left[  0,T\right]  $, for every $\varepsilon\in(0,1]$, for
every $y_{0}\in L^{2}(0,L)$ and for every $\overline{y}_{0}\in L^{2}(0,L)$.
\end{lemma}

\begin{proof}
We multiply the first equation of (\ref{lambda1}) by
$2x\Delta$ and then integrate over $[0,L]$. By integrating by parts
and using the boundary conditions of
\eqref{lambda1}, we get%
\begin{align}
&  \frac{d}{dt}\int_{0}^{L}x\Delta^{2}dx+3\int_{0}^{L}\Delta_{x}%
^{2}dx\nonumber\\
 = &\int_{0}^{L}\Delta^{2}dx+\Phi_{1}\times\left(
-2\int_{0}^{L}x\alpha
\Delta\Delta_{x}dx+4\int_{0}^{L}x\overline{\alpha}\Delta\Delta_{x}dx+2\int
_{0}^{L}xz\Delta\Delta_{x}dx\right. \nonumber\\
&  +2\int_{0}^{L}\overline{\alpha}\Delta^{2}dx+2\int_{0}^{L}z\Delta
^{2}dx-2\int_{0}^{L}x\Delta\overline{\alpha}\left(  z-\overline{z}\right)
_{x}dx\label{2xlambda}\\
&  \left.  -2\int_{0}^{L}x\Delta\overline{\alpha}_{x}\left(  z-\overline
{z}\right)  dx-2\int_{0}^{L}x\Delta z_{x}\left(  z-\overline{z}\right)
dx-2\int_{0}^{L}x\Delta\overline{z}\left(  z-\overline{z}\right)
_{x}dx\right) \nonumber\\
&  -\left(  \Phi_{1}-\Phi_{2}\right)  \int_{0}^{L}2x\Delta\left(  \overline
{z}_{x}\overline{\alpha}+\overline{\alpha}_{x}\overline{z}+\overline{z}%
_{x}\overline{z}+\overline{\alpha}_{x}\overline{\alpha}\right)  dx.\nonumber
\end{align}
Note that $\alpha\left(  t,0\right)  =\alpha\left(  t,L\right)  =0$,
by the continuous Sobolev embedding $H_{0}^{1}\left(  0,L\right)
\subset C^{0}([0,L])$ and
Poincar\'{e} inequality$,$ there exists $C=C(L)>0$ such that%
\[
2\left\vert \int_{0}^{L}x\alpha\Delta\Delta_{x}dx\right\vert \leq C\left\Vert
\alpha_{x}\right\Vert _{L^{2}\left(  0,L\right)  }\int_{0}^{L}\left\vert
x\Delta\Delta_{x}\right\vert dx.
\]
Thus,%
\begin{align}
2\left\vert \int_{0}^{L}x\alpha\Delta\Delta_{x}dx\right\vert  &  \leq
C\left\Vert \alpha_{x}\right\Vert _{L^{2}\left(  0,L\right)  }\left\Vert
\Delta_{x}\right\Vert _{L^{2}\left(  0,L\right)  }\left\Vert x\Delta
\right\Vert _{L^{2}\left(  0,L\right)  }\nonumber\\
&  \leq\frac{1}{2}\left\Vert \Delta_{x}\right\Vert _{L^{2}\left(  0,L\right)
}^{2}+\frac{1}{2}\left(  C\left\Vert \alpha_{x}\right\Vert _{L^{2}\left(
0,L\right)  }\left\Vert x\Delta\right\Vert _{L^{2}\left(  0,L\right)
}\right)  ^{2}\nonumber\\
&  \leq\frac{1}{2}\int_{0}^{L}\Delta_{x}^{2}dx+C\left\Vert \alpha
_{x}\right\Vert _{L^{2}\left(  0,L\right)  }^{2}\int_{0}^{L}x\Delta^{2}dx.
\label{2xlambda1}%
\end{align}
Similarly,%
\begin{equation}
4\left\vert \int_{0}^{L}x\overline{\alpha}\Delta\Delta_{x}dx\right\vert
\leq\frac{1}{2}\int_{0}^{L}\Delta_{x}^{2}dx+C\left\Vert \overline{\alpha}%
_{x}\right\Vert _{L^{2}\left(  0,L\right)  }^{2}\int_{0}^{L}x\Delta^{2}dx,
\label{2xlambda2}%
\end{equation}%
\begin{equation}
2\left\vert \int_{0}^{L}xz\Delta\Delta_{x}dx\right\vert \leq\frac{1}{2}%
\int_{0}^{L}\Delta_{x}^{2}dx+C\left\Vert z_{x}\right\Vert _{L^{2}\left(
0,L\right)  }^{2}\int_{0}^{L}x\Delta^{2}dx. \label{2xlambda3}%
\end{equation}
Note that $\overline{\alpha}\left(  t,0\right)
=\overline{\alpha}\left( t,L\right) =0$, by the continuous Sobolev
embedding $H_{0}^{1}\left( 0,L\right)  \subset
C^{0}([0,L])$ and Poincar\'{e} inequality, we have%
\begin{equation}
2\left\vert \int_{0}^{L}\overline{\alpha}\Delta^{2}dx\right\vert \leq
C\left\Vert \overline{\alpha}_{x}\right\Vert _{L^{2}\left(  0,L\right)  }%
\int_{0}^{L}\Delta^{2}dx. \label{ineqbaralphaDelta2}%
\end{equation}
 From \eqref{ineqbaralphaDelta2} and Lemma~16 in \cite{Coron-Crepeau} with
$a:=\min\left\{  \displaystyle\frac{1}{\sqrt{2}}C^{-\frac{1}{2}}\left\Vert \overline
{\alpha}_{x}\right\Vert _{L^{2}\left(  0,L\right)  }^{-\frac{1}{2}},L\right\}
,$ there exists $C=C(L)>0$ such that%
\begin{equation}
2\left\vert \int_{0}^{L}\overline{\alpha}\Delta^{2}dx\right\vert \leq\frac
{1}{4}\int_{0}^{L}\Delta_{x}^{2}dx+C\left(  \left\Vert \overline{\alpha}%
_{x}\right\Vert _{L^{2}\left(  0,L\right)  }^{\frac{3}{2}}+\left\Vert
\overline{\alpha}_{x}\right\Vert _{L^{2}\left(  0,L\right)  }\right)  \int
_{0}^{L}x\Delta^{2}dx. \label{2xlambda4}%
\end{equation}
Similarly, we have%
\begin{equation}
2\left\vert \int_{0}^{L}z\Delta^{2}dx\right\vert \leq\frac{1}{4}\int_{0}%
^{L}\Delta_{x}^{2}dx+C\left(  \left\Vert z_{x}\right\Vert _{L^{2}\left(
0,L\right)  }^{\frac{3}{2}}+\left\Vert z_{x}\right\Vert _{L^{2}\left(
0,L\right)  }\right)  \int_{0}^{L}x\Delta^{2}dx. \label{2xlambda5}%
\end{equation}
By Lemma~16 in \cite{Coron-Crepeau}, there exists $C=C(L)>0$ such
that
\begin{equation}
\int_{0}^{L}\Delta^{2}dx\leq\frac{1}{2}\int_{0}^{L}\Delta_{x}^{2}dx+C\int
_{0}^{L}x\Delta^{2}dx. \label{2xlambda6}%
\end{equation}
We have%
\begin{align}
2\left\vert \int_{0}^{L}x\Delta\overline{\alpha}\left(  z-\overline{z}\right)
_{x}dx\right\vert  &  \leq C\left\Vert \overline{\alpha}_{x}\right\Vert
_{L^{2}\left(  0,L\right)  }\left\vert \int_{0}^{L}x\Delta\left(
z-\overline{z}\right)  _{x}dx\right\vert \nonumber\\
&  \leq C\left\Vert \overline{\alpha}_{x}\right\Vert _{L^{2}\left(
0,L\right)  }\left(  \int_{0}^{L}x^{2}\Delta^{2}dx\right)  ^{\frac{1}{2}%
}\left\Vert \left(  z-\overline{z}\right)  _{x}\right\Vert _{L^{2}\left(
0,L\right)  }\nonumber\\
&  \leq C\left\Vert \overline{\alpha}_{x}\right\Vert _{L^{2}\left(
0,L\right)  }\left(  \int_{0}^{L}x\Delta^{2}dx\right)  ^{\frac{1}{2}%
}\left\Vert \left(  z-\overline{z}\right)  _{x}\right\Vert _{L^{2}\left(
0,L\right)  }. \label{Sobolev embedding and Holder inequality}%
\end{align}
Similarly, we can obtain%
\begin{equation}
2\left\vert \int_{0}^{L}x\Delta\overline{\alpha}_{x}\left(  z-\overline
{z}\right)  dx\right\vert \leq C\left\Vert \left(  z-\overline{z}\right)
_{x}\right\Vert _{L^{2}\left(  0,L\right)  }\left(  \int_{0}^{L}x\Delta
^{2}dx\right)  ^{\frac{1}{2}}\left\Vert \overline{\alpha}_{x}\right\Vert
_{L^{2}\left(  0,L\right)  }, \label{2xlambda7}%
\end{equation}%
\begin{equation}
2\left\vert \int_{0}^{L}x\Delta z_{x}\left(  z-\overline{z}\right)
dx\right\vert \leq C\left\Vert \left(  z-\overline{z}\right)  _{x}\right\Vert
_{L^{2}\left(  0,L\right)  }\left(  \int_{0}^{L}x\Delta^{2}dx\right)
^{\frac{1}{2}}\left\Vert z_{x}\right\Vert _{L^{2}\left(  0,L\right)  },
\label{2xlambda8}%
\end{equation}
\begin{equation}
2\left\vert \int_{0}^{L}x\Delta\overline{z}\left(  z-\overline{z}\right)
_{x}dx\right\vert \leq C\left\Vert \overline{z}_{x}\right\Vert _{L^{2}\left(
0,L\right)  }\left(  \int_{0}^{L}x\Delta^{2}dx\right)  ^{\frac{1}{2}%
}\left\Vert \left(  z-\overline{z}\right)  _{x}\right\Vert _{L^{2}\left(
0,L\right)  }. \label{2xlambda9}%
\end{equation}
Moreover, we have%
\begin{align}
&  \left\vert \int_{0}^{L}2x\Delta\left(  \overline{z}_{x}\overline{\alpha
}+\overline{\alpha}_{x}\overline{z}+\overline{z}_{x}\overline{z}%
+\overline{\alpha}_{x}\overline{\alpha}\right)  dx\right\vert \nonumber\\
=& 2\left\vert \int_{0}^{L}x\Delta\left(  \overline{z}_{x}+\overline{\alpha
}_{x}\right)  \left(  \overline{z}+\overline{\alpha}\right)  dx\right\vert \nonumber\\
\leq & 2\left\Vert \overline{z}+\overline{\alpha}\right\Vert
_{L^{\infty }(0,L)}\int_{0}^{L}\left\vert x\Delta\left(
\overline{z}+\overline{\alpha
}\right)  _{x}\right\vert dx\nonumber\\
\leq &  2\sqrt{L}\left\Vert
\overline{z}+\overline{\alpha}\right\Vert
_{L^{\infty}(0,L)}\left\Vert \left(
\overline{z}+\overline{\alpha}\right) _{x}\right\Vert _{L^{2}\left(
0,L\right)  }\left(  \int_{0}^{L}x\Delta ^{2}dx\right)
^{\frac{1}{2}} \label{+1}.
\end{align}
Then, using the Gagliardo-Nirenberg inequality \eqref{15.1*}, it
follows from (\ref{+1}) that
\begin{align}
&  \left\vert \int_{0}^{L}2x\Delta\left(  \overline{z}_{x}\overline{\alpha
}+\overline{\alpha}_{x}\overline{z}+\overline{z}_{x}\overline{z}%
+\overline{\alpha}_{x}\overline{\alpha}\right)  dx\right\vert \nonumber\\
\leq & C\left\Vert \overline{z}+\overline{\alpha}\right\Vert _{L^{2}%
(0,L)}^{\frac{1}{2}}\left\Vert \left(  \overline{z}+\overline{\alpha}\right)
_{x}\right\Vert _{L^{2}\left(  0,L\right)  }^{\frac{3}{2}}\left(  \int_{0}%
^{L}x\Delta^{2}dx\right)  ^{\frac{1}{2}}. \label{2xlambda10}%
\end{align}
Thus, using (\ref{2xlambda}) to (\ref{2xlambda10}), we get
\begin{align}
&  \frac{d}{dt}\int_{0}^{L}x\Delta^{2}dx+\frac{1}{2}\int_{0}^{L}\Delta_{x}%
^{2}dx\nonumber\\
\leq & C\left(  1+\Phi_{1}\left(  \left\Vert \alpha_{x}\right\Vert
_{L^{2}\left(  0,L\right)  }^{2}+\left\Vert \overline{\alpha}_{x}\right\Vert
_{L^{2}\left(  0,L\right)  }^{2}+\left\Vert z_{x}\right\Vert _{L^{2}\left(
0,L\right)  }^{2}\right)  \right)  \int_{0}^{L}x\Delta_{{}}^{2}dx\nonumber\\
&  +C\left[  \Phi_{1}\left(  \left\Vert \overline{\alpha}_{x}\right\Vert
_{L^{2}\left(  0,L\right)  }+\left\Vert z_{x}\right\Vert _{L^{2}\left(
0,L\right)  }+\left\Vert \overline{z}_{x}\right\Vert _{L^{2}\left(
0,L\right)  }\right)  \left\Vert \left(  z-\overline{z}\right)  _{x}%
\right\Vert _{L^{2}\left(  0,L\right)  }\right. \nonumber\\
&  \left.  +\left\vert \Phi_{1}-\Phi_{2}\right\vert \left\Vert
\overline{z}+\overline{\alpha}\right\Vert _{L^{2}(0,L)}^{\frac{1}{2}%
}\left\Vert \left(  \overline{z}+\overline{\alpha}\right)  _{x}\right\Vert
_{L^{2}\left(  0,L\right)  }^{\frac{3}{2}}\right]  \left(  \int_{0}^{L}%
x\Delta^{2}dx\right)^{\frac{1}{2}}. \label{step1}%
\end{align}
In particular,
\begin{align*}
&  \frac{d}{dt}\int_{0}^{L}x\Delta^{2}dx\\
\leq & C\left(  1+\Phi_{1}\left(  \left\Vert \alpha_{x}\right\Vert
_{L^{2}\left(  0,L\right)  }^{2}+\left\Vert \overline{\alpha}_{x}\right\Vert
_{L^{2}\left(  0,L\right)  }^{2}+\left\Vert z_{x}\right\Vert _{L^{2}\left(
0,L\right)  }^{2}\right)  \right)  \int_{0}^{L}x\Delta_{{}}^{2}dx\\
&  +C\left[  \Phi_{1}\left(  \left\Vert \overline{\alpha}_{x}\right\Vert
_{L^{2}\left(  0,L\right)  }+\left\Vert z_{x}\right\Vert _{L^{2}\left(
0,L\right)  }+\left\Vert \overline{z}_{x}\right\Vert _{L^{2}\left(
0,L\right)  }\right)  \left\Vert \left(  z-\overline{z}\right)  _{x}%
\right\Vert _{L^{2}\left(  0,L\right)  }\right. \\
& \left.  +\left\vert \Phi_{1}-\Phi_{2}\right\vert \left\Vert
\overline{z}+\overline{\alpha}\right\Vert _{L^{2}(0,L)}^{\frac{1}{2}%
}\left\Vert \left(  \overline{z}+\overline{\alpha}\right)  _{x}\right\Vert
_{L^{2}\left(  0,L\right)  }^{\frac{3}{2}}\right]  \left(  \int_{0}^{L}%
x\Delta^{2}dx\right)  ^{\frac{1}{2}}.
\end{align*}
Then, by Lemma 17 in \cite{Coron-Crepeau}, we get
\begin{equation}
\int_{0}^{L}x\Delta^{2}dx\leq W \label{W1}, \quad \forall t\in\left[
0,T\right]  
\end{equation}
with
\begin{align}
W:=& \ 3 C^2 \Big[\int_0^T\Big(\left(  \left\Vert \Phi_1\overline
{\alpha}_{x}\right\Vert _{L^{2}\left(  0,L\right)  }+\left\Vert\Phi_1 z_{x}%
\right\Vert _{L^{2}\left(  0,L\right)  }+\left\Vert \Phi_1\overline{z}%
_{x}\right\Vert _{L^{2}\left(  0,L\right)  }\right)\left\Vert \left(
z-\overline{z}\right)  _{x}\right\Vert _{L^{2}\left(  0,L\right)  }\nonumber\\
&\qquad  +\left\vert
\Phi_{1}-\Phi_{2}\right\vert \left\Vert \overline{z}+\overline{\alpha
}\right\Vert _{L^{2}(0,L)}^{\frac{1}{2}}\left\Vert \left(  \overline{z}+\overline{\alpha
}\right)  _{x}\right\Vert _{L^{2}\left(  0,L\right)  }^{\frac{3}{2}}\Big)dt\Big]^2\nonumber\\
&\times \exp\left[  C\left(  T+\left\Vert \sqrt{\Phi_{1}}\alpha
_{x}\right\Vert _{L^{2}(L^{2})}^{2}+\left\Vert \sqrt{\Phi_{1}}\overline
{\alpha}_{x}\right\Vert _{L^{2}(L^{2})}^{2}+\left\Vert \sqrt{\Phi_{1}}%
z_{x}\right\Vert _{L^{2}(L^{2})}^{2}\right)  \right].
\label{W}%
\end{align}
Now integrating (\ref{step1}) over $[0,T]$ and using (\ref{W1}), we have%
\begin{align*}
&  \int_{0}^{L}x\Delta^{2}\left(  T,x\right)  dx+\frac{1}{2}\int_{0}^{T}%
\int_{0}^{L}\Delta_{x}^{2}dxdt\\
\leq & \  C\int_{0}^{T}\left(  1+\Phi_{1}\left(  \left\Vert \alpha
_{x}\right\Vert _{L^{2}\left(  0,L\right)  }^{2}+\left\Vert \overline{\alpha
}_{x}\right\Vert _{L^{2}\left(  0,L\right)  }^{2}+\left\Vert z_{x}\right\Vert
_{L^{2}\left(  0,L\right)  }^{2}\right)  \right)  dtW\\
& +C\int_{0}^{T}\left[  \Phi_{1}\left(  \left\Vert \overline
{\alpha}_{x}\right\Vert _{L^{2}\left(  0,L\right)  }+\left\Vert z_{x}%
\right\Vert _{L^{2}\left(  0,L\right)  }+\left\Vert \overline{z}%
_{x}\right\Vert _{L^{2}\left(  0,L\right)  }\right)  \left\Vert \left(
z-\overline{z}\right)  _{x}\right\Vert _{L^{2}\left(  0,L\right)  }\right. \\
& \left.  +\left\vert \Phi_{1}-\Phi_{2}\right\vert \left\Vert
\overline{z}+\overline{\alpha}\right\Vert _{L^{2}(0,L)}^{\frac{1}{2}%
}\left\Vert \left(  \overline{z}+\overline{\alpha}\right)  _{x}\right\Vert
_{L^{2}\left(  0,L\right)  }^{\frac{3}{2}}\right]  dtW^{\frac{1}{2}}.
\end{align*}
Then it follows that%
\begin{align}
&  \frac{1}{2}\int_{0}^{T}\int_{0}^{L}\Delta_{x}^{2}dxdt\\
\leqslant
&\ C\int_0^T \Big(1+\Phi_1\left(\|\alpha_x\|^2_{L^2(0,L)}+\|\overline\alpha_x\|^2_{L^2(0,L)}+\|z_x\|^2_{L^2(0,L)}\right)\Big)\, dt W+\frac{1}{2}W
\nonumber\\
&+\frac{1}{2}\Big[C\int_0^T\Big(\Phi_1\left(  \left\Vert \overline
{\alpha}_{x}\right\Vert _{L^{2}\left(  0,L\right)  }+\left\Vert z_{x}%
\right\Vert _{L^{2}\left(  0,L\right)  }+\left\Vert \overline{z}%
_{x}\right\Vert _{L^{2}\left(  0,L\right)  }\right)\left\Vert \left(
z-\overline{z}\right)  _{x}\right\Vert _{L^{2}\left(  0,L\right)  }\nonumber\\
&\qquad +\left\vert \Phi_{1}-\Phi_{2}\right\vert \left\Vert
\overline{z}+\overline{\alpha }\right\Vert
_{L^{2}(0,L)}^{\frac{1}{2}}\left\Vert \left(
\overline{z}+\overline{\alpha }\right)  _{x}\right\Vert
_{L^{2}\left(  0,L\right)  }^{\frac{3}{2}}\Big)dt\Big]^2. \label{+2}
\end{align}
Hence,  combining (\ref{+2}) with (\ref{W}), we obtain
\begin{align}
&  \int_{0}^{T}\int_{0}^{L}\Delta_{x}^{2}dxdt\nonumber\\
\leqslant
&\ C\Big[\int_0^T\Big(\left(  \left\Vert \Phi_1\overline
{\alpha}_{x}\right\Vert _{L^{2}\left(  0,L\right)  }+\left\Vert\Phi_1 z_{x}%
\right\Vert _{L^{2}\left(  0,L\right)  }+\left\Vert \Phi_1\overline{z}%
_{x}\right\Vert _{L^{2}\left(  0,L\right)  }\right)\left\Vert \left(
z-\overline{z}\right)  _{x}\right\Vert _{L^{2}\left(  0,L\right)  }\nonumber\\
&\qquad +\left\vert
\Phi_{1}-\Phi_{2}\right\vert \left\Vert \overline{z}+\overline{\alpha
}\right\Vert _{L^{2}(0,L)}^{\frac{1}{2}}\left\Vert \left(  \overline{z}+\overline{\alpha
}\right)  _{x}\right\Vert _{L^{2}\left(  0,L\right)  }^{\frac{3}{2}}\Big)dt\Big]^2\nonumber\\
&\times \exp\left[  C\left(  T+\left\Vert \sqrt{\Phi_{1}}\alpha
_{x}\right\Vert _{L^{2}(L^{2})}^{2}+\left\Vert \sqrt{\Phi_{1}}\overline
{\alpha}_{x}\right\Vert _{L^{2}(L^{2})}^{2}+\left\Vert \sqrt{\Phi_{1}}%
z_{x}\right\Vert _{L^{2}(L^{2})}^{2}\right)  \right].
\label{L^2-lambda_x^2}%
\end{align}
We multiply the first equation of (\ref{lambda1}) by $\Delta$ and integrate
over $[0,L]$. Using the boundary conditions of (\ref{lambda1}) and
integrations by parts, we get%
\begin{align}
&  \frac{1}{2}\frac{d}{dt}\int_{0}^{L}\Delta^{2}dx+\frac{1}{2}\Delta_{x}%
^{2}\left(  t,0\right) \nonumber\\
 &= \Phi_{1}\times\left(  -\int_{0}^{L}\alpha\Delta_{x}\Delta dx+2\int_{0}%
^{L}\overline{\alpha}\Delta_{x}\Delta dx+\int_{0}^{L}z\Delta_{x}\Delta
dx\right. \nonumber\\
&\quad - \int_{0}^{L}\overline{\alpha}\left(  z-\overline{z}\right)
_{x}\Delta dx-\int_{0}^{L}\overline{\alpha}_{x}\left(
z-\overline{z}\right)
\Delta dx\nonumber\\
& \quad \left.  -\int_{0}^{L}z_{x}\left(  z-\overline{z}\right)
\Delta dx-\int_{0}^{L}\overline{z}\left(  z-\overline{z}\right)
_{x}\Delta
dx\right) \label{time lambda}\\
& \quad  -\left(  \Phi_{1}-\Phi_{2}\right)  \int_{0}^{L}\Delta\left(
\overline
{z}_{x}\overline{\alpha}+\overline{\alpha}_{x}\overline{z}+\overline{z}%
_{x}\overline{z}+\overline{\alpha}_{x}\overline{\alpha}\right)  dx.\nonumber
\end{align}
It can be readily checked that%
\begin{align}
\left\vert \int_{0}^{L}\alpha\Delta_{x}\Delta dx\right\vert  &  \leq\frac
{1}{2}\int_{0}^{L}\Delta_{x}^{2}dx+\frac{1}{2}\int_{0}^{L}\Delta^{2}\alpha
^{2}dx\nonumber\\
&  \leq\frac{1}{2}\int_{0}^{L}\Delta_{x}^{2}dx+C\left\Vert \alpha
_{x}\right\Vert _{L^{2}(0,L)}^{2}\int_{0}^{L}\Delta^{2}dx.
\label{time lambda1}%
\end{align}
Similarly, we have%
\begin{equation}
\left\vert 2\int_{0}^{L}\overline{\alpha}\Delta_{x}\Delta dx\right\vert
\leq\frac{1}{2}\int_{0}^{L}\Delta_{x}^{2}dx+C\left\Vert \overline{\alpha}%
_{x}\right\Vert _{L^{2}(0,L)}^{2}\int_{0}^{L}\Delta^{2}dx,
\label{time lambda2}%
\end{equation}
and%
\begin{equation}
\left\vert \int_{0}^{L}z\Delta_{x}\Delta dx\right\vert \leq\frac{1}{2}\int
_{0}^{L}\Delta_{x}^{2}dx+C\left\Vert z_{x}\right\Vert _{L^{2}(0,L)}^{2}%
\int_{0}^{L}\Delta^{2}dx. \label{time lambda3}%
\end{equation}
Similarly to (\ref{Sobolev embedding and Holder inequality}), we get the
following inequalities%
\begin{align}
\left\vert \int_{0}^{L}\overline{\alpha}\left(  z-\overline{z}\right)
_{x}\Delta dx\right\vert  &  \leq C\left\Vert \overline{\alpha}_{x}\right\Vert
_{L^{2}\left(  0,L\right)  }\left\Vert \left(  z-\overline{z}\right)
_{x}\right\Vert _{L^{2}\left(  0,L\right)  }\left(  \int_{0}^{L}\Delta
^{2}dx\right)  ^{\frac{1}{2}},\label{time lambda4}\\
\left\vert \int_{0}^{L}\overline{\alpha}_{x}\left(  z-\overline{z}\right)
\Delta dx\right\vert  &  \leq C\left\Vert \left(  z-\overline{z}\right)
_{x}\right\Vert _{L^{2}\left(  0,L\right)  }\left\Vert \overline{\alpha}%
_{x}\right\Vert _{L^{2}\left(  0,L\right)  }\left(  \int_{0}^{L}\Delta
^{2}dx\right)  ^{\frac{1}{2}}, \label{time lambda5}\\
\left\vert \int_{0}^{L}z_{x}\left(  z-\overline{z}\right)  \Delta
dx\right\vert  &  \leq C\left\Vert \left(  z-\overline{z}\right)
_{x}\right\Vert _{L^{2}\left(  0,L\right)  }\left\Vert z_{x}\right\Vert
_{L^{2}\left(  0,L\right)  }\left(  \int_{0}^{L}\Delta^{2}dx\right)
^{\frac{1}{2}},\label{time lambda6}\\
\left\vert \int_{0}^{L}\overline{z}\left(  z-\overline{z}\right)  _{x}\Delta
dx\right\vert  &  \leq C\left\Vert \overline{z}_{x}\right\Vert _{L^{2}\left(
0,L\right)  }\left\Vert \left(  z-\overline{z}\right)  _{x}\right\Vert
_{L^{2}\left(  0,L\right)  }\left(  \int_{0}^{L}\Delta^{2}dx\right)
^{\frac{1}{2}}. \label{time lambda7}%
\end{align}
Moreover, for the last term on the right-hand side of (\ref{time lambda}),
using the same argument as for (\ref{2xlambda10}), we have%
\begin{align}
&  \left\vert \int_{0}^{L}\Delta\left(  \overline{z}_{x}\overline{\alpha
}+\overline{\alpha}_{x}\overline{z}+\overline{z}_{x}\overline{z}%
+\overline{\alpha}_{x}\overline{\alpha}\right)  dx\right\vert \nonumber\\
&  \leq C\left\Vert \overline{z}+\overline{\alpha}\right\Vert _{L^{2}%
(0,L)}^{\frac{1}{2}}\left\Vert \left(  \overline{z}+\overline{\alpha}\right)
_{x}\right\Vert _{L^{2}\left(  0,L\right)  }^{\frac{3}{2}}\left(  \int_{0}%
^{L}\Delta^{2}dx\right)  ^{\frac{1}{2}}. \label{time lambda8}%
\end{align}
Hence, by (\ref{time lambda}) to (\ref{time lambda8}), we deduce that%
\begin{align*}
&  \frac{1}{2}\frac{d}{dt}\int_{0}^{L}\Delta^{2}dx\\
\leq &  \ \frac{3}{2}\int_{0}^{L}\Delta_{x}^{2}dx+C\Phi_{1}\left(  \left\Vert
\alpha_{x}\right\Vert _{L^{2}\left(  0,L\right)  }^{2}+\left\Vert
\overline{\alpha}_{x}\right\Vert _{L^{2}\left(  0,L\right)  }^{2}+\left\Vert
z_{x}\right\Vert _{L^{2}\left(  0,L\right)  }^{2}\right)  \int_{0}^{L}%
\Delta^{2}dx\\
&  +C\left[   \Phi_{1}\left(  2\left\Vert \overline{\alpha
}_{x}\right\Vert _{L^{2}\left(  0,L\right)  }+\left\Vert z_{x}\right\Vert
_{L^{2}\left(  0,L\right)  }+\left\Vert \overline{z}_{x}\right\Vert
_{L^{2}\left(  0,L\right)  }\right)  \left\Vert \left(  z-\overline{z}\right)
_{x}\right\Vert _{L^{2}\left(  0,L\right)  }  \right. \\
&  \left.  +\left\vert \Phi_{1}-\Phi_{2}\right\vert \left\Vert
\overline{z}+\overline{\alpha}\right\Vert _{L^{2}(0,L)}^{\frac{1}{2}%
}\left\Vert \left(  \overline{z}+\overline{\alpha}\right)  _{x}\right\Vert
_{L^{2}\left(  0,L\right)  }^{\frac{3}{2}}\right]  \left(  \int_{0}^{L}%
\Delta^{2}dx\right)  ^{\frac{1}{2}}.
\end{align*}
Therefore, by (\ref{L^2-lambda_x^2}) and Lemma 17 in \cite{Coron-Crepeau}, we
get that, for every $t\in\left[  0,T\right]  ,$%
\begin{align*}
&\|\Delta(t,\cdot)\|_{L^2(0,L)}\\
\leqslant
&\ \Big[\int_0^T\Big(\Phi_1\left(  \left\Vert \overline
{\alpha}_{x}\right\Vert _{L^{2}\left(  0,L\right)  }+\left\Vert z_{x}%
\right\Vert _{L^{2}\left(  0,L\right)  }+\left\Vert \overline{z}%
_{x}\right\Vert _{L^{2}\left(  0,L\right)  }\right)\left\Vert \left(
z-\overline{z}\right)  _{x}\right\Vert _{L^{2}\left(  0,L\right)  }\\
&\qquad +\left\vert
\Phi_{1}-\Phi_{2}\right\vert \left\Vert \overline{z}+\overline{\alpha
}\right\Vert _{L^{2}(0,L)}^{\frac{1}{2}}\left\Vert \left(  \overline{z}+\overline{\alpha
}\right)  _{x}\right\Vert _{L^{2}\left(  0,L\right)  }^{\frac{3}{2}}\Big)dt\Big]^2\\
&\times \exp\left[  C\left(  1+\left\Vert \sqrt{\Phi_{1}}\alpha
_{x}\right\Vert _{L^{2}(L^{2})}^{2}+\left\Vert \sqrt{\Phi_{1}}\overline
{\alpha}_{x}\right\Vert _{L^{2}(L^{2})}^{2}+\left\Vert \sqrt{\Phi_{1}}%
z_{x}\right\Vert _{L^{2}(L^{2})}^{2}\right)  \right].
\end{align*}
This completes the proof of Lemma \ref{estimate of delta}.
\end{proof}

Now we are in a position to prove the following proposition on the global
Lipschitzianity of the map $R(t)$. With our notation, we have%
\[
R(t)y_{0}-R(t)\overline{y}_{0}=\alpha\left(  t,\cdot\right)  -\overline
{\alpha}\left(  t,\cdot\right)  =\Delta.
\]

\begin{proposition}
\label{global Lipschitz}Let $T>0$. There exists $\varepsilon_{0}\in(0,1]$ and
$\tilde C :(0,\varepsilon_{0}]\rightarrow(0,+\infty)$ such that
\begin{gather}
\label{estRliptendvers0}\left\Vert \Delta\right\Vert _{L^{2}\left(
0,L\right)  }\leq\tilde C\left(  \varepsilon\right)  \left\Vert y_{0}%
-\overline{y}_{0}\right\Vert _{L^{2}\left(  0,L\right)
},\quad \forall\,\overline {y}_{0},y_{0}\in L^{2}\left(  0,L\right)
,\forall t\in\left[  0,T\right]  ,
\forall\varepsilon\in(0,\varepsilon_{0}] ,\\
\tilde C \left(  \varepsilon\right)  \rightarrow0\text{ as }\varepsilon
\rightarrow0^{+}.
\end{gather}

\end{proposition}

\begin{proof}
Let
\[
\Delta_{\max}:=\sup_{\substack{t\in\left[  0,T\right]  \\\varepsilon\in\left(
0,1\right]  }}\left\Vert \Delta(t,\cdot)\right\Vert _{L^{2}\left(  0,L\right)
}.
\]
Let us point out that, by Lemma \ref{delta bounded}, $\Delta_{\max}<+\infty$.
We claim that%
\begin{equation}
\Delta_{\max}\leq\varepsilon C\left(  \left\Vert y_{0}-\overline{y}%
_{0}\right\Vert _{L^{2}\left(  0,L\right)  }+\Delta_{\max}\right)
,\quad\forall\,\overline{y}_{0},y_{0}\in L^{2}\left(  0,L\right)  . \label{delta-max}%
\end{equation}
Then let $\varepsilon$ be small enough such that $1-\varepsilon C>0,$ we
obtain%
\[
\Delta_{\max}\leq\frac{\varepsilon C}{1-\varepsilon C}\left\Vert
y_{0}-\overline{y}_{0}\right\Vert _{L^{2}\left(  0,L\right)
},\quad\forall \,\overline{y}_{0},y_{0}\in L^{2}\left(  0,L\right)
.
\]
Consequently, we get
\[
\left\Vert \Delta\right\Vert _{L^{2}\left(  0,L\right)  }\leq\frac{\varepsilon
C}{1-\varepsilon C}\left\Vert y_{0}-\overline{y}_{0}\right\Vert _{L^{2}\left(
0,L\right)  },\quad\forall \,t\in\left[  0,T\right]  ,\quad \forall\overline{y}_{0}%
,y_{0}\in L^{2}\left(  0,L\right)  ,
\]
and the result follows. Hence, in order to prove Proposition
\ref{global Lipschitz}, we only need to show that (\ref{delta-max}) holds in
the following cases:

\begin{itemize}
\item[(i)] $\left\Vert \overline{y}_{0}\right\Vert _{L^{2}(0,L)}$ $,\left\Vert
y_{0}\right\Vert _{L^{2}(0,L)}$ $\geq\varepsilon,$ and $\left\Vert
\overline{y}\right\Vert _{L^{2}(0,L)}$ $,\left\Vert y\right\Vert _{L^{2}%
(0,L)}$ $\geq\varepsilon,$ $\forall t\in\left[  0,T\right]  ;$

\item[(ii)] $\left\Vert \overline{y}_{0}\right\Vert _{L^{2}(0,L)}$
$,\left\Vert y_{0}\right\Vert _{L^{2}(0,L)}$ $\geq\varepsilon,$ there exists
$\tau\in\left[  0,T\right]  $ such that $\left\Vert \overline{y}\left(
\tau,\cdot\right)  \right\Vert _{L^{2}(0,L)}=\varepsilon,$ and $\left\Vert
y\left(  t,\cdot\right)  \right\Vert _{L^{2}(0,L)}\geq\varepsilon,\forall
t\in\left[  0,T\right]  ;$

\item[(iii)] $\left\Vert \overline{y}_{0}\right\Vert _{L^{2}(0,L)}$
$,\left\Vert y_{0}\right\Vert _{L^{2}(0,L)}$ $\geq\varepsilon,$ there exists
$\tau,\varsigma\in\left[  0,T\right]  ,\varsigma>\tau,$ such that $\left\Vert
\overline{y}\left(  \tau,\cdot\right)  \right\Vert _{L^{2}(0,L)}=\varepsilon,$
and $\left\Vert y\left(  \varsigma,\cdot\right)  \right\Vert _{L^{2}%
(0,L)}=\varepsilon;$

\item[(iv)] $\left\Vert \overline{y}_{0}\right\Vert _{L^{2}(0,L)}%
\leq\varepsilon$ and $\left\Vert y\right\Vert _{L^{2}(0,L)}$ $\geq
\varepsilon,$ $\forall t\in\left[  0,T\right]  ;$

\item[(v)] $\left\Vert \overline{y}_{0}\right\Vert _{L^{2}(0,L)}%
\leq\varepsilon$ $,\left\Vert y_{0}\right\Vert _{L^{2}(0,L)}$ $\geq
\varepsilon,$ and there exists $\tau\in\left[  0,T\right]  $ such that
$\left\Vert y\left(  \tau,\cdot\right)  \right\Vert _{L^{2}(0,L)}%
=\varepsilon;$

\item[(vi)] $\left\Vert \overline{y}_{0}\right\Vert _{L^{2}(0,L)}%
\leq\varepsilon$ $,\left\Vert y_{0}\right\Vert _{L^{2}(0,L)}$ $\leq
\varepsilon.$

By Lemma \ref{estimate of delta}, for every $ t\in\left[  0,T\right]  $, we have%
\begin{align}
&  \left\Vert \Delta(t,\cdot)\right\Vert _{L^{2}\left(  0,L\right)
}\nonumber\\
\leq&  \ \int_{0}^{T}\left[  \Phi_{1}\left(  \left\Vert \overline{\alpha}%
_{x}\right\Vert _{L^{2}\left(  0,L\right)  }+\left\Vert z_{x}\right\Vert
_{L^{2}\left(  0,L\right)  }+\left\Vert \overline{z}_{x}\right\Vert
_{L^{2}\left(  0,L\right)  }\right)  \left\Vert \left(  z-\overline{z}\right)
_{x}\right\Vert _{L^{2}\left(  0,L\right)  }\right. \nonumber\\
&   \left.  \qquad  +\left\vert \Phi_{1}-\Phi_{2}\right\vert \left\Vert
\overline{z}+\overline{\alpha}\right\Vert _{L^{2}(0,L)}^{\frac{1}{2}%
}\left\Vert \left(  \overline{z}+\overline{\alpha}\right)  _{x}\right\Vert
_{L^{2}\left(  0,L\right)  }^{\frac{3}{2}}\right]  dt\nonumber\\
& \times\exp\left[  C\left(  1+\left\Vert \sqrt{\Phi_{1}}\alpha
_{x}\right\Vert _{L^{2}(L^{2})}^{2}+\left\Vert \sqrt{\Phi_{1}}\overline
{\alpha}_{x}\right\Vert _{L^{2}(L^{2})}^{2}+\left\Vert \sqrt{\Phi_{1}}%
z_{x}\right\Vert _{L^{2}(L^{2})}^{2}\right)  \right] . \label{*}%
\end{align}
 Furthermore, by using H\"{o}lder's inequality and Lemma \ref{LE1}, we have%
\begin{align}
&  \int_{0}^{T}\Phi_{1}\left(  \left\Vert \overline{\alpha}_{x}\right\Vert
_{L^{2}\left(  0,L\right)  }+\left\Vert z_{x}\right\Vert _{L^{2}\left(
0,L\right)  }+\left\Vert \overline{z}_{x}\right\Vert _{L^{2}\left(
0,L\right)  }\right)  \left\Vert \left(  z-\overline{z}\right)  _{x}%
\right\Vert _{L^{2}\left(  0,L\right)  }dt\nonumber\\
\leq&  \ \left\Vert \left(  z-\overline{z}\right)  _{x}\right\Vert _{L^{2}%
(L^{2})}\left(  \left\Vert \Phi_{1}\overline{\alpha}_{x}\right\Vert
_{L^{2}(L^{2})}+\left\Vert \Phi_{1}z_{x}\right\Vert _{L^{2}(L^{2})}+\left\Vert
\Phi_{1}\overline{z}_{x}\right\Vert _{L^{2}(L^{2})}\right) \nonumber\\
\leq&\  C\left\Vert y_{0}-\overline{y}_{0}\right\Vert _{L^{2}\left(
0,L\right)  }\left(  \left\Vert \Phi_{1}\overline{\alpha}_{x}\right\Vert
_{L^{2}(L^{2})}+\left\Vert \Phi_{1}z_{x}\right\Vert _{L^{2}(L^{2})}+\left\Vert
\Phi_{1}\overline{z}_{x}\right\Vert _{L^{2}(L^{2})}\right)  . \label{**}%
\end{align}
Applying the mean value theorem, noticing that $\left\Vert \alpha
-\overline{\alpha}\right\Vert _{L^{2}\left(  0,L\right)  }\leq\Delta_{\max
},\forall t\in\left[  0,T\right]  ,$  and by Lemma~\ref{LE3},
\[
\left\Vert z-\overline{z}\right\Vert _{L^{2}\left(  0,L\right)  }%
\leq\left\Vert y_{0}-\overline{y}_{0}\right\Vert _{L^{2}\left(  0,L\right)
},\quad \forall t\in\left[  0,T\right]  ,
\]
we get%
\begin{align}
\left\vert \Phi_{1}-\Phi_{2}\right\vert  &  \leq\left\vert \Phi_{\varepsilon
}^{\prime}\left(  \theta\right)  \right\vert \left\vert \left\Vert
z+\alpha\right\Vert _{L^{2}\left(  0,L\right)  }-\left\Vert \overline
{z}+\overline{\alpha}\right\Vert _{L^{2}\left(  0,L\right)  }\right\vert
\nonumber\\
&  \leq\left\vert \Phi_{\varepsilon}^{\prime}\left(  \theta\right)
\right\vert \left\Vert \left(  z+\alpha\right)  -\left(  \overline
{z}+\overline{\alpha}\right)  \right\Vert _{L^{2}\left(  0,L\right)
}\nonumber\\
&  \leq\left\vert \Phi_{\varepsilon}^{\prime}\left(  \theta\right)
\right\vert \left(  \left\Vert z-\overline{z}\right\Vert _{L^{2}\left(
0,L\right)  }+\left\Vert \alpha-\overline{\alpha}\right\Vert _{L^{2}\left(
0,L\right)  }\right) \nonumber\\
&  \leq\left\vert \Phi_{\varepsilon}^{\prime}\left(  \theta\right)
\right\vert \left(  \left\Vert y_{0}-\overline{y}_{0}\right\Vert
_{L^{2}\left(  0,L\right)  }+\Delta_{\max}\right)  , \label{***}%
\end{align}
where%
\[
\theta=\theta(t)\in\left(  \min\{\left\Vert z+\alpha\right\Vert _{L^{2}\left(
0,L\right)  },\left\Vert \overline{z}+\overline{\alpha}\right\Vert
_{L^{2}\left(  0,L\right)  }\},\max\{\left\Vert z+\alpha\right\Vert
_{L^{2}\left(  0,L\right)  },\left\Vert \overline{z}+\overline{\alpha
}\right\Vert _{L^{2}\left(  0,L\right)}  \}\right)  .
\]
Thus, combining (\ref{*}), (\ref{**}) and (\ref{***}), we arrive at
\begin{align*}
&  \left\Vert \Delta(t,\cdot)\right\Vert _{L^{2}\left(  0,L\right)  }\\
\leq& \ \left(  C\left\Vert y_{0}-\overline{y}_{0}\right\Vert _{L^{2}\left(
0,L\right)  }\left(  \left\Vert \Phi_{1}\overline{\alpha}_{x}\right\Vert
_{L^{2}(L^{2})}+\left\Vert \Phi_{1}z_{x}\right\Vert _{L^{2}(L^{2})}^{{}%
}+\left\Vert \Phi_{1}\overline{z}_{x}\right\Vert _{L^{2}(L^{2})}^{{}}\right)
\right. \\
&  \left.  +\left(  \left\Vert y_{0}-\overline{y}_{0}\right\Vert
_{L^{2}\left(  0,L\right)  }+\Delta_{\max}\right)  \int_{0}^{T}\left\vert
\Phi_{\varepsilon}^{\prime}\left(  \theta\right)  \right\vert \left\Vert
\overline{z}+\overline{\alpha}\right\Vert _{L^{2}(0,L)}^{\frac{1}{2}%
}\left\Vert \left(  \overline{z}+\overline{\alpha}\right)  _{x}\right\Vert
_{L^{2}\left(  0,L\right)  }^{\frac{3}{2}}dt\right) \\
&  \times\exp\left(  C\left(  1+\left\Vert \sqrt{\Phi_{1}}%
\alpha_{x}\right\Vert _{L^{2}(L^{2})}^{2}+\left\Vert \sqrt{\Phi_{1}}%
\overline{\alpha}_{x}\right\Vert _{L^{2}(L^{2})}^{2}+\left\Vert \sqrt{\Phi
_{1}}z_{x}\right\Vert _{L^{2}(L^{2})}^{2}\right)  \right)  .
\end{align*}
Consequently, we obtain%
\begin{equation}%
\begin{array}
[c]{l}%
\Delta_{\max}\leq\left[  C\left\Vert y_{0}-\overline{y}_{0}\right\Vert
_{L^{2}\left(  0,L\right)  }\left(  \left\Vert \Phi_{1}\overline{\alpha}%
_{x}\right\Vert _{L^{2}(L^{2})}+\left\Vert \Phi_{1}z_{x}\right\Vert
_{L^{2}(L^{2})}+\left\Vert \Phi_{1}\overline{z}_{x}\right\Vert _{L^{2}(L^{2}%
)}\right)  \right. \\
\text{ }\ \ \ \ \ \text{\ }\quad\left.  +\left(  \left\Vert y_{0}-\overline{y}%
_{0}\right\Vert _{L^{2}\left(  0,L\right)  }+\Delta_{\max}\right)  \displaystyle\int
_{0}^{T}\left\vert \Phi_{\varepsilon}^{\prime}\left(  \theta\right)
\right\vert \left\Vert \overline{y}\right\Vert _{L^{2}(0,L)}^{\frac{1}{2}%
}\left\Vert \overline{y}_{x}\right\Vert _{L^{2}\left(  0,L\right)  }^{\frac
{3}{2}}dt\right] \\
\text{ \ \ \ \ \ \ }\quad \times\exp\left(  C\left(  1+\left\Vert \sqrt{\Phi_{1}}%
\alpha_{x}\right\Vert _{L^{2}(L^{2})}^{2}+\left\Vert \sqrt{\Phi_{1}}%
\overline{\alpha}_{x}\right\Vert _{L^{2}(L^{2})}^{2}+\left\Vert \sqrt{\Phi
_{1}}z_{x}\right\Vert _{L^{2}(L^{2})}^{2}\right)  \right)  .
\end{array}
\label{main estimate}%
\end{equation}
 For case (i),  by (\ref{Phi1=Phi2=0}), we have $\Phi_{1}=\Phi
_{2}=0, \forall t\in[0,T]$, it follows directly from (\ref{*}) that
\[
\Delta_{\max}=0.
\]

 For case (ii),  by (\ref{Phi1=Phi2=0}), we have
\begin{equation}
\Phi_{1}\equiv0,\quad \forall t\in\left[  0,T\right]  . \label{phi1=0}%
\end{equation}
In view of Lemma \ref{decreasingL2norm}, we have
\begin{equation}
\left\Vert \overline{y}\left(  t,\cdot\right)  \right\Vert _{L^{2}(0,L)}%
\geq\varepsilon,\quad \forall t\in\left[  0,\tau\right]  , \label{lowerbary}%
\end{equation}
and%
\begin{equation}
\left\Vert \overline{y}\left(  t,\cdot\right)  \right\Vert _{L^{2}(0,L)}%
\leq\left\Vert \overline{y}\left(  \tau,\cdot\right)  \right\Vert
_{L^{2}(0,L)}=\varepsilon,\quad \forall t\in\left[  \tau,T\right]  . \label{ii-2}%
\end{equation}
Consequently, it follows from (\ref{Phi1=Phi2=0}) and
(\ref{lowerbary}) that
\[
\Phi_{2}\equiv0,\quad \forall t\in\left[  0,\tau\right]  .
\]
 From \eqref{derivative of Phi}, (\ref{main estimate}), \eqref{phi1=0},
\eqref{lowerbary} and \eqref{ii-2}, we get that
\begin{equation}
\Delta_{\max}\leq \exp\left(  C\right)  \left(  \left\Vert
y_{0}-\overline{y}_{0}\right\Vert _{L^{2}\left(  0,L\right)
}+\Delta_{\max}\right)  \int_{\tau}^{T}\frac
{C}{\varepsilon}\varepsilon^{\frac{1}{2}}\left\Vert \overline{y}%
_{x}\right\Vert _{L^{2}\left(  0,L\right)  }^{\frac{3}{2}}dt  . \label{ii-1}%
\end{equation}

 From now on, we assume that $\varepsilon\in(0,\eta]$, where
$\eta>0$ be chosen as in Lemma \ref{y_x-L^2}. Thanks to Lemma \ref{y_x-L^2} and
(\ref{ii-2}), we have
\begin{equation}
\left\Vert \overline{y}_{x}\left(  t,\cdot\right)  \right\Vert _{L^{2}%
(0,L)}\leq\frac{C}{\sqrt{t-\tau}}\left\Vert \overline{y}\left(  \tau
,\cdot\right)  \right\Vert _{L^{2}\left(  0,L\right)  }=\frac{C}{\sqrt{t-\tau
}}\varepsilon,\quad \forall t\in\left[  \tau,T\right]  . \label{ii-3}%
\end{equation}

Replacing (\ref{ii-3}) into (\ref{ii-1}), we obtain%
\begin{align*}
\Delta_{\max}  &  \leq \varepsilon C\left(  \left\Vert y_{0}-\overline{y}_{0}\right\Vert
_{L^{2}\left(  0,L\right)  }+\Delta_{\max}\right)  \int_{\tau}%
^{T}\frac{1}{\left(  t-\tau\right)  ^{\frac{3}{4}}}dt\\
&  \leq\varepsilon C\left(  \left\Vert y_{0}-\overline{y}_{0}\right\Vert
_{L^{2}\left(  0,L\right)  }+\Delta_{\max}\right)  .
\end{align*}
 For case (iii), by (\ref{Phi1=Phi2=0}) and Lemma
\ref{decreasingL2norm}, we have
\begin{equation}
\Phi_{1}=0,\ \forall t\in\left[  0,\varsigma\right]  , \label{iii-Phi1=0}%
\end{equation}%
\begin{equation}
\Phi_{2}=0,\ \forall t\in\left[  0,\tau\right]  , \label{iii-Phi2=0}%
\end{equation}
and (\ref{ii-2}) still holds. In particular,
\begin{equation}
\left\Vert \overline{y}\left(  \varsigma,\cdot\right)  \right\Vert
_{L^{2}\left(  0,L\right)  }\leq\varepsilon. \label{iii-ybar-epsilon}%
\end{equation}
It follows from Lemma \ref{LE1},  (\ref{iii-Phi1=0}) and
(\ref{iii-ybar-epsilon}) that%
\begin{align}
\left\Vert \Phi_{1}\overline{z}_{x}\right\Vert _{L^{2}(0,T;L^{2}\left(
0,L\right)  )}^{{}}  &  =\left\Vert \Phi_{1}\overline{z}_{x}\right\Vert
_{L^{2}(\varsigma,T;L^{2}\left(  0,L\right)  )}^{{}}\nonumber\\
&  \leq\left\Vert \overline{z}_{x}\right\Vert
_{L^{2}(\varsigma,T;L^{2}\left( 0,L\right)  )}^{{}}\leq C\left\Vert
\overline{y}\left(  \varsigma ,\cdot\right)  \right\Vert
_{L^{2}\left(  0,L\right)  }\leq \varepsilon C,
\label{iii-Phi1-zbarx}\\
\left\Vert \Phi_{1}z_{x}\right\Vert _{L^{2}(0,T;L^{2}\left(  0,L\right)
)}^{{}}  &  =\left\Vert \Phi_{1}z_{x}\right\Vert _{L^{2}(\varsigma
,T;L^{2}\left(  0,L\right)  )}^{{}}\nonumber\\
&  \leq\left\Vert z_{x}\right\Vert _{L^{2}(\varsigma,T;L^{2}\left(
0,L\right)  )}^{{}}\leq C\left\Vert y\left(  \varsigma,\cdot\right)
\right\Vert _{L^{2}\left(  0,L\right)  }=\varepsilon C, \label{iii-Phi1-zx}%
\end{align}
and%
\begin{equation}
\left\Vert \sqrt{\Phi_{1}}z_{x}\right\Vert _{L^{2}(0,T;L^{2}\left(
0,L\right)  )}^{{}}\leq \varepsilon C. \label{iii-sqrt-Phi1-zx}%
\end{equation}
By Remark \ref{R1}, (\ref{iii-Phi1=0}), (\ref{iii-ybar-epsilon})
and (\ref{iii-Phi1-zbarx}), we have
\begin{align}
\left\Vert \Phi_{1}\overline{\alpha}_{x}\right\Vert _{L^{2}(0,T;L^{2}\left(
0,L\right)  )}^{{}}  &  \leq\left\Vert \Phi_{1}\overline{y}_{x}\right\Vert
_{L^{2}(0,T;L^{2}\left(  0,L\right)  )}^{{}}+\left\Vert \Phi_{1}\overline
{z}_{x}\right\Vert _{L^{2}(0,T;L^{2}\left(  0,L\right)  )}^{{}}\nonumber\\
&  \leq\left\Vert \overline{y}_{x}\right\Vert _{L^{2}(\varsigma,T;L^{2}\left(
0,L\right)  )}^{{}}+\left\Vert \Phi_{1}\overline{z}_{x}\right\Vert
_{L^{2}(0,T;L^{2}\left(  0,L\right)  }^{{}}\nonumber\\
&  \leq \varepsilon C, \label{bounded1}%
\end{align}
and
\begin{equation}
\left\Vert \sqrt{\Phi_{1}}\overline{\alpha}_{x}\right\Vert _{L^{2}%
(0,T;L^{2}\left(  0,L\right)  )}^{{}}\leq \varepsilon C.
\label{iii-sqrt-Phi1-alphax}%
\end{equation}
Similarly, we obtain
\begin{equation}
\left\Vert \sqrt{\Phi_{1}}\alpha_{x}\right\Vert _{L^{2}(0,T;L^{2}\left(
0,L\right)  )}^{{}}\leq \varepsilon C. \label{bounded2}%
\end{equation}
 Moreover, for this case,  (\ref{ii-3}) still holds. Now
it follows from (\ref{derivative of Phi}), (\ref{main estimate}),
(\ref{ii-2}), (\ref{ii-3}), (\ref{iii-Phi1=0}), (\ref{iii-Phi2=0}),
(\ref{iii-Phi1-zbarx})
to (\ref{bounded2}) that%
\begin{align*}
\Delta_{\max}  &  \leq C\left(  C\varepsilon\left\Vert y_{0}-\overline{y}%
_{0}\right\Vert _{L^{2}\left(  0,L\right)  }+\varepsilon\left(  \left\Vert y_{0}%
-\overline{y}_{0}\right\Vert _{L^{2}\left(  0,L\right)  }+\Delta_{\max
}\right)\int_{\tau}^{T}\frac{1}{\left(  t-\tau\right)  ^{\frac
{3}{4}}}dt\right) \\
&  \leq\varepsilon C\left(  \left\Vert y_{0}-\overline{y}_{0}\right\Vert
_{L^{2}\left(  0,L\right)  }+\Delta_{\max}\right)  .
\end{align*}
 For case (iv), by (\ref{Phi1=Phi2=0}), we have $\Phi_{1}%
\equiv0,\forall t\in\left[  0,T\right]  .$ It follows from
(\ref{main estimate}) that
\begin{equation}
\Delta_{\max}\leq C \left(  \left\Vert
y_{0}-\overline{y}_{0}\right\Vert _{L^{2}\left(  0,L\right)
}+\Delta_{\max}\right)  \int_{0}^{T}\left\vert \Phi^{\prime}\left(
\theta\right)  \right\vert \left\Vert \overline
{y}\right\Vert _{L^{2}(0,L)}^{\frac{1}{2}}\left\Vert \overline{y}%
_{x}\right\Vert _{L^{2}\left(  0,L\right)  }^{\frac{3}{2}}dt  .
\label{iv-1}%
\end{equation}
By Lemma \ref{decreasingL2norm}, we have%
\begin{equation}
\left\Vert \overline{y}\left(  t,\cdot\right)  \right\Vert _{L^{2}(0,L)}%
\leq\left\Vert \overline{y}_{0}\right\Vert _{L^{2}(0,L)}\leq\varepsilon
,\quad \forall t\in\left[  0,T\right]  . \label{iv-2}%
\end{equation}
Moreover, thanks to Lemma \ref{y_x-L^2}, we have%
\begin{equation}
\left\Vert \overline{y}_{x}\left(  t,\cdot\right)  \right\Vert _{L^{2}%
(0,L)}\leq\frac{C}{\sqrt{t}}\left\Vert \overline{y}_{0}\right\Vert
_{L^{2}\left(  0,L\right)  }\leq\frac{C}{\sqrt{t}}\varepsilon,\quad\forall
t\in\left[  0,T\right]  . \label{iv-3}%
\end{equation}
Then it follows from (\ref{derivative of Phi}), (\ref{iv-1}) to
(\ref{iv-3}) that%
\begin{align*}
\Delta_{\max}  &  \leq  \varepsilon C\left(  \left\Vert y_{0}-\overline{y}%
_{0}\right\Vert _{L^{2}\left(  0,L\right)  }+\Delta_{\max}\right)
\int_{0}^{T}\frac{1}{t^{\frac{3}{4}}}dt \\
&  \leq\varepsilon C\left(  \left\Vert y_{0}-\overline{y}_{0}\right\Vert
_{L^{2}\left(  0,L\right)  }+\Delta_{\max}\right)  .
\end{align*}
 For case (v), similarly to case (iii), we have
\begin{align}
\left\Vert \Phi_{1}\overline{\alpha}_{x}\right\Vert
_{L^{2}(0,T;L^{2}\left( 0,L\right)  )}^{{}}  &  \leq \varepsilon
C,\quad \left\Vert \Phi_{1}z_{x}\right\Vert
_{L^{2}(0,T;L^{2}\left(  0,L\right)  )}^{{}}\leq \varepsilon C,\label{v1}\\
\left\Vert \Phi_{1}\overline{z}_{x}\right\Vert
_{L^{2}(0,T;L^{2}\left( 0,L\right)  )}^{{}}  &  \leq \varepsilon
C,\quad \left\Vert \sqrt{\Phi_{1}}\alpha _{x}\right\Vert
_{L^{2}(0,T;L^{2}\left(  0,L\right)  )}^{{}}\leq \varepsilon C,\\
\left\Vert \sqrt{\Phi_{1}}\overline{\alpha}_{x}\right\Vert _{L^{2}%
(0,T;L^{2}\left(  0,L\right)  )}^{{}}  &  \leq \varepsilon
C,\quad \left\Vert
\sqrt{\Phi_{1}}z_{x}\right\Vert _{L^{2}(0,T;L^{2}\left(  0,L\right)  )}^{{}%
}\leq \varepsilon C, \label{v2}%
\end{align}
Moreover, (\ref{iv-2}) and (\ref{iv-3}) still hold. Thanks to
(\ref{derivative of Phi}), (\ref{iv-2}) and (\ref{iv-3}), we have%
\begin{equation}
\int_{0}^{T}\left\vert \Phi^{\prime}\left(  \theta\right)  \right\vert
\left\Vert \overline{y}\right\Vert _{L^{2}(0,L)}^{\frac{1}{2}}\left\Vert
\overline{y}_{x}\right\Vert _{L^{2}\left(  0,L\right)  }^{\frac{3}{2}}dt\leq
\varepsilon C\int_{0}^{T}\frac{1}{t^{\frac{3}{4}}}dt. \label{v3}%
\end{equation}
Then, by (\ref{main estimate}), (\ref{v1}) to (\ref{v3}), we obtain%
\[
\Delta_{\max}\leq \varepsilon C\left(  \left\Vert y_{0}-\overline{y}%
_{0}\right\Vert _{L^{2}\left(  0,L\right)  }+\Delta_{\max}\right)  .
\]
 For the last case (vi), (\ref{iv-2})-(\ref{v3}) hold, and
(\ref{delta-max}) follows. Above all, we have proved (\ref{delta-max}) for
all the cases (i)-(vi), which completes the proof of
Proposition~\ref{global Lipschitz}.
\end{itemize}
\end{proof}

\subsubsection{\bigskip Smoothness of the semigroup}

\begin{lemma}
\label{lemma0} Let $\varepsilon>0$ and $T>0$ be given. Then the nonlinear map
$\mathcal{S}(t)$ defined by the unique solution of (\ref{new system}) is of class
$C^{3}$ from $L^{2}\left(  0,L\right)  $ to $C\left(  \left[  0,T\right]
;L^{2}\left(  0,L\right)  \right)  $. Moreover, its derivative ${\mathcal{S}}^{(1)}$ at
$y_{0}\in L^{2}\left(  0,L\right)  $ is given by
\begin{equation}
{\mathcal{S}}^{(1)}(y_{0})(h):=\mathcal{K}^{(1)}(y)(h),\, \forall
h\in L^{2}(0,L),
\end{equation}
where $\mathcal{K}^{(1)}(y)(h)$ is defined by the following system \eqref{e1}
with $y=\mathcal{S}(y_{0})$.
\end{lemma}

\begin{equation}
\label{e1}%
\begin{cases}
\Delta_{t}+\Delta_{x}+\Delta_{xxx}+\Phi^{\prime}_{\varepsilon}(\|y\|_{L^{2}%
(0,L)})\displaystyle\frac{\int_{0}^{L} y\Delta\, dx}{\|y\|_{L^{2}(0,L)}}yy_{x}
+\Phi_{\varepsilon}(\|y\|_{L^{2}(0,L)})(y\Delta_{x}+\Delta y_{x})=0,\\
\Delta(t,0)=\Delta(t,L)=0,\\
\Delta_{x}(t,L)=0,\\
\Delta(0,x)=h(x),
\end{cases}
\end{equation}

\begin{proof}
We refer to \cite{Zhang1995} and \cite[Theorem 5.4]{Bona-Sun-Zhang} for a
detailed argument in related circumstances.
\end{proof}

\subsubsection{\bigskip Center manifold}

Combining \cite[Remark 2.3]{Minh-Wu2004}, Corollary
\ref{Corollary on spectrum} and Proposition \ref{global Lipschitz}, we are in
a position to apply \cite[Theorem 2.19]{Minh-Wu2004} and \cite[Theorem
2.28]{Minh-Wu2004}. This gives, if $\varepsilon>0$ is small enough which will
be always assumed from now on, the existence of an invariant center manifold
for (\ref{new system}) which is of class $C^{3}$. (In fact this center
manifold is called a center-unstable manifold in \cite[Theorem 2.19]%
{Minh-Wu2004}; however, in our situation, with the notations of
\cite{Minh-Wu2004}, $P_{2}(t)$, $t\in\mathbb{R}$, are trivial
projections and then the name of center manifold can be adopted: see
\cite[Remark 2.20]{Minh-Wu2004}.) More precisely, there exists a map
$g:M\rightarrow M^{\bot}$ of class $C^{3}$ satisfying $g(0)=0$ and
$g^{\prime}(0)=0$, such that, if
\[
G:=\left\{  x_{1}+g\left(  x_{1}\right)  :x_{1}\in M\right\}  ,
\]
then, for every $y_{0}\in G$ and for every $t\in[0,+\infty)$, $\mathcal{S}%
(t)y_{0}\in G$. Moreover, Theorem \ref{main result} holds, if (and only if),
\begin{gather}
\label{dyncenter}\mathcal{S}(t)y_{0}\rightarrow0 \text{ as } t \rightarrow
+\infty, \, \forall y_{0}\in G \text{ such that } \|y_{0}\|_{L^{2}(0,L)}
\text{ is small enough.}%
\end{gather}
(For this last statement, see (2.42) in \cite{Minh-Wu2004}.) We
prove \eqref{dyncenter} in the next section.

\section{Dynamic on the center manifold}

\label{sec4} In this section, we prove \eqref{dyncenter}, which concludes the
proof of Theorem \ref{main result}.

\begin{proof}
Let $y_{0}\in G$. Let, for $t\in[0,+\infty)$, $y(t)(x):=y(t,x):=(\mathcal{S}%
(t)y_{0})(x)$. We write
\begin{equation}
y(t,x)=p(t)\varphi(x)+y^{\star}(t,x)\text{,} \label{decomposition-y}%
\end{equation}
where $\phi(x)$ is defined in (\ref{phi}) and $y^{\star}(t,x)\in
M^{\bot}.$ By \eqref{phi} and \eqref{projection-on-M}, we have, at
least if $\|y(t)\|_{L^{2}(0,L)}$ is small enough which will be
always assumed in this proof,
\begin{align}
\frac{dp(t)}{dt}  &  =\int_{0}^{L}y_{t}\left(  t,x\right)  \varphi
(x)dx=\int_{0}^{L}\left(  -y_{x}-yy_{x}-y_{xxx}\right)  \varphi
(x)dx\nonumber\\
&  =\int_{0}^{L}y(t,x)\varphi(x)dx-\int_{0}^{L}y(t,x)y_{x}(t,x)\varphi
(x)dx+\int_{0}^{L}y(t,x)\varphi_{xxx}(x)dx\nonumber\\
&  =\frac{1}{2}\int_{0}^{L}y^{2}(t,x)\varphi_{x}(x)dx. \label{1}%
\end{align}
We can also obtain the system for $y^{\star}(t,x)$ as the following

\begin{equation}
\left\{
\begin{array}
[c]{l}%
y_{t}^{\star}+y_{x}^{\star}+(I-P)yy_{x}+y_{xxx}^{\star}=0,\\
y^{\star}(t,0)=y^{\star}(t,L)=0,\\
y_{x}^{\star}(t,L)=0.
\end{array}
\right.  \label{y star}%
\end{equation}
 It follows from \eqref{decomposition-y} that
\begin{align*}
yy_{x}  &  =\left(  p(t)\varphi(x)+y^{\star}(t,x)\right)  \left(
p(t)\varphi_{x}(x)+y_{x}^{\star}(t,x)\right) \\
&  =p^{2}(t)\varphi(x)\varphi_{x}(x)+p(t)y^{\star}(t,x)\varphi_{x}%
(x)+p(t)\varphi(x)y_{x}^{\star}(t,x)+y^{\star}(t,x)y_{x}^{\star}(t,x).
\end{align*}
Consequently, we have
\begin{align}
(I-P)yy_{x} =
&\ p^{2}(t)\varphi(x)\varphi_{x}(x)+p(t)y^{\star}(t,x)\varphi
_{x}(x)+p(t)\varphi(x)y_{x}^{\star}(t,x)+y^{\star}(t,x)y_{x}^{\star}(t,x)\nonumber\\
&  -p^{2}(t)\varphi(x)\int_{0}^{L}\varphi^{2}(x)\varphi_{x}(x)dx-p(t)\varphi
(x)\int_{0}^{L}y^{\star}(t,x)\varphi(x)\varphi_{x}(x)dx\nonumber\\
&
-p(t)\varphi(x)\int_{0}^{L}\varphi^{2}(x)y_{x}^{\star}(t,x)dx-\varphi
(x)\int_{0}^{L}\varphi(x)y^{\star}(t,x)y_{x}^{\star}(t,x)dx.\label{(I-P)yyx}
\end{align}
 By using (\ref{phi}) and integrations by parts, we have
\begin{align}
\int_{0}^{L}\varphi^{2}(x)\varphi_{x}(x)dx  &  =0,\label{phi^2phix}\\
\int_{0}^{L}y^{\star}(t,x)\varphi(x)\varphi_{x}(x)dx  &  =-\frac{1}{2}\int
_{0}^{L}\varphi^{2}(x)y_{x}^{\star}(t,x)dx,\label{y*phiphix}\\
\int_{0}^{L}\varphi(x)y^{\star}(t,x)y_{x}^{\star}(t,x)dx& =-\frac{1}{2}\int
_{0}^{L}\varphi_{x}(x)\left(  y^{\star}(t,x)\right)  ^{2}dx.\label{phiy*y*x}
\end{align}
It can be deduced from (\ref{(I-P)yyx}), (\ref{phi^2phix}),
(\ref{y*phiphix}) and (\ref{phiy*y*x}) that
\begin{align}
(I-P)yy_{x} = &\
p^{2}(t)\varphi(x)\varphi_{x}(x)+p(t)y^{\star}(t,x)\varphi
_{x}(x)+p(t)\varphi(x)y_{x}^{\star}(t,x)+y^{\star}(t,x)y_{x}^{\star}(t,x)\nonumber\\
&  -\frac{1}{2}p(t)\varphi(x)\int_{0}^{L}\varphi^{2}(x)y_{x}^{\star
}(t,x)dx+\frac{1}{2}\varphi(x)\int_{0}^{L}\varphi_{x}(x)\left(  y^{\star
}(t,x)\right)  ^{2}dx.\label{(I-P)yyxfinal}
\end{align}
According to the existence and smoothness of the center manifold, we can set%
\begin{equation}
y^{\star}(t,x)=a(x)p^{2}(t)+O(p^{3}(t)),\, \text{ as }|p(t)|\rightarrow0.
\label{ystar}%
\end{equation}
Then, by using (\ref{y star}), (\ref{(I-P)yyxfinal}) and by
comparing the coefficients of $p^{2}(t)$,
we obtain%
\begin{equation}
\left\{
\begin{array}
[c]{l}%
a_{x}(x)+a_{xxx}(x)+\varphi(x)\varphi_{x}(x)=0,\\
a(0)=a(L)=0,\\
a_{x}(L)=0.
\end{array}
\right.  \label{a}%
\end{equation}
The solution of (\ref{a}) is
\[
a(x)=C_{1}+C_{2}\cos x-\frac{1}{3}\sin x+\frac{1}{6\pi}x\sin x+\frac{1}{36\pi
}\cos(2x),
\]
where
\begin{equation}
C_{1}+C_{2}=-\frac{1}{36\pi}. \label{C1}%
\end{equation}
Note that $y^{\star}(t,x)\in M^{\bot},$ we have
\[
\int_{0}^{L}a(x)\varphi(x)dx=0,
\]
i.e.%
\[
C_{1}\frac{2\pi}{\sqrt{3\pi}}+C_{2}\frac{-\pi}{\sqrt{3\pi}}+\frac{1}{6\pi
}\times\frac{-3\pi}{2\sqrt{3\pi}}=0,
\]
which leads to
\begin{equation}
2\pi C_{1}-\pi C_{2}-\frac{1}{4}=0. \label{C2}%
\end{equation}
Combining (\ref{C1}) and (\ref{C2}), we get
\begin{align*}
C_{1}    =\frac{2}{27\pi},\quad
C_{2}    =-\frac{11}{108\pi}.
\end{align*}
Therefore,
\begin{equation}
a(x)=\frac{2}{27\pi}-\frac{11}{108\pi}\cos x-\frac{1}{3}\sin x+\frac{1}{6\pi
}x\sin x+\frac{1}{36\pi}\cos(2x). \label{a(x)}%
\end{equation}
Combining  (\ref{decomposition-y}), (\ref{1}), (\ref{ystar}) and
(\ref{a(x)}), we obtain
\begin{align*}
\frac{dp(t)}{dt}  &  =\frac{1}{2}\int_{0}^{L}\left(  p(t)\varphi
(x)+a(x)p^{2}(t)+O(p^{3}(t))\right)  ^{2}\varphi_{x}(x)dx\\
&  =p^{3}(t)\int_{0}^{L}a(x)\varphi(x)\varphi_{x}(x)dx+O(p^{4}(t))\\
&  =\frac{p^{3}(t)}{3\pi}\left(  -\frac{1}{3}\pi+\frac{1}{6\pi}\pi^{2}\right)
+O(p^{4}(t))\\
&  =-\frac{{p^{3}(t)}}{18}+O(p^{4}(t)),\, \text{ as }\  |p(t)|\rightarrow0.
\end{align*}
This concludes the proof of \eqref{dyncenter} and the proof of Theorem
\ref{main result}.
\end{proof}

\section*{Acknowledgement}
We thank Thierry Gallay, G\'{e}rard Iooss, Lionel Rosier and Bing-Yu Zhang for
their valuable advices during the preparation of this work.

\end{document}